\documentclass[10pt,reqno,oneside]{amsart}
\title[Global existence for the Navier-Stokes equations in the half-space]{Global weak solutions of the Navier-Stokes equations for intermittent initial data in half-space}

\date{\today}

\author[Z.~Bradshaw]{Zachary Bradshaw}
\address{Department of Mathematical Sciences, University of Arkansas, Fayetteville, AR 72701}
\email{zb002@uark.edu}

\author[I.~Kukavica]{Igor Kukavica}
\address{Department of Mathematics, University of Southern California, Los Angeles, CA 90089}
\email{kukavica@usc.edu}

\author[W.~S.~O\.za\'nski]{Wojciech S. O\.za\'nski}
\address{Department of Mathematics, University of Southern California, Los Angeles, CA 90089}
\email{ozanski@usc.edu}

  \chardef\forshowkeys=0
  \chardef\refcheck=0
  \chardef\showllabel=0
  \chardef\sketches=0

\usepackage{enumitem}

\usepackage{fancyhdr}
\usepackage{comment,caption}

\usepackage{mathtools}

\DeclareMathOperator*{\esssup}{ess\,sup}

\ifnum\forshowkeys=1
  
  \usepackage[notref,notcite,color]{showkeys}
\fi

\usepackage[margin=1in]{geometry}
\usepackage{amsmath, amsthm, amssymb}
\usepackage{times}
\usepackage{graphicx}
\usepackage[usenames,dvipsnames,svgnames,table]{xcolor}
\usepackage{marginnote}
\usepackage[unicode,breaklinks=true,colorlinks=true,linkcolor=blue,urlcolor=blue,citecolor=blue]{hyperref}

\usepackage{tikz}

\ifnum\refcheck=1
  \usepackage{refcheck}
\fi

\newcommand{\EQ}[1]{\begin{equation}\begin{split} #1 \end{split}\end{equation}}

\begin{document}

\ifnum\showllabel=1
 \def\llabel#1{\marginnote{\color{lightgray}\rm\small(#1)}[-0.0cm]\notag}
\else
 \def\llabel#1{\notag}
\fi

\newcommand{\norm}[1]{\left\|#1\right\|}
\newcommand{\nnorm}[1]{\lVert #1\rVert}
\newcommand{\abs}[1]{\left|#1\right|}
\newcommand{\NORM}[1]{|\!|\!| #1|\!|\!|}

\newtheorem{Theorem}{Theorem}[section]
\newtheorem{Corollary}[Theorem]{Corollary}
\newtheorem{Definition}[Theorem]{Definition}
\newtheorem{Proposition}[Theorem]{Proposition}
\newtheorem{Lemma}[Theorem]{Lemma}
\newtheorem{Remark}[Theorem]{Remark}

\def\theequation{\thesection.\arabic{equation}}
\numberwithin{equation}{section}

\definecolor{myblue}{rgb}{.8, .8, 1}

\newlength\mytemplen
\newsavebox\mytempbox

\makeatletter
\newcommand\mybluebox{%
    \@ifnextchar[
       {\@mybluebox}%
       {\@mybluebox[0pt]}}

\def\@mybluebox[#1]{%
    \@ifnextchar[
       {\@@mybluebox[#1]}%
       {\@@mybluebox[#1][0pt]}}

\def\@@mybluebox[#1][#2]#3{
    \sbox\mytempbox{#3}%
    \mytemplen\ht\mytempbox
    \advance\mytemplen #1\relax
    \ht\mytempbox\mytemplen
    \mytemplen\dp\mytempbox
    \advance\mytemplen #2\relax
    \dp\mytempbox\mytemplen
    \colorbox{myblue}{\hspace{1em}\usebox{\mytempbox}\hspace{1em}}}

\makeatother

\def\tt{{\tilde t}}
\newcommand*{\re}{\ensuremath{\mathrm{{\mathbb R}e\,}}}
\newcommand*{\im}{\ensuremath{\mathrm{{\mathbb I}m\,}}}
\newcommand{\un}{\underline{n}}
\newcommand{\on}{\overline{n}}

\def\Helm{\mathrm{H}}
\def\nonloc{\mathrm{ nonloc}}
\def\harm{\mathrm{harm}}
\def\li{\mathrm{li}}
\def\HH{{\mathbb R}^{3}_{+}}
\def\biglinem{\vskip0.5truecm\par==========================\par\vskip0.5truecm}
\def\inon#1{\hbox{\ \ \ \ \ }\hbox{#1}}                
\def\and{\text{\indeq and\indeq}}
\def\indeqtimes{\quad{}\quad{}{\quad{}\quad{}\quad{}\quad{}\quad{}\quad{}\times}}
\def\onon#1{\text{~~on~$#1$}}
\def\inin#1{\text{~~in~$#1$}}
\def\startnewsection#1#2{ \section{#1}\label{#2}\setcounter{equation}{0}}   
\def\Tr{\mathop{\rm Tr}\nolimits}    
\def\div{\mathop{\rm div}\nolimits}
\def\curl{\mathop{\rm curl}\nolimits}
\def\dist{\mathop{\rm dist}\nolimits}  
\def\supp{\mathop{\rm supp}\nolimits}
\def\indeq{\quad{}}           

\def\colr{\color{red}}
\def\colb{\color{black}}
\def\coly{\color{lightgray}}
\definecolor{colorgggg}{rgb}{0.1,0.5,0.3}
\definecolor{colorllll}{rgb}{0.0,0.7,0.0}
\definecolor{mypurple}{rgb}{0.8,0.0,0.2}
\definecolor{colorhhhh}{rgb}{0.3,0.75,0.4}
\definecolor{colorpppp}{rgb}{0.7,0.0,0.2}
\definecolor{coloroooo}{rgb}{0.5,0.0,0.0}
\definecolor{colorqqqq}{rgb}{0.1,0.7,0}
\def\colg{\color{colorgggg}}
\def\collg{\color{colorllll}}
\def\coleo{\color{colorpppp}}
\def\cole{\color{black}}
\def\colu{\color{blue}}
\def\colw{\color{coloraaaa}}

\def\comma{ {\rm ,\qquad{}} }            
\def\commaone{ {\rm ,\quad{}} }          
\def\les{\lesssim}
\def\nts#1{\hbox{\color{red}\hbox{\bf ~#1~}}} 
\def\ntsf#1{\footnote{\color{colorgggg}\hbox{#1}}} 
\def\ntsfred#1{\footnote{\color{red}\hbox{#1}}} 
\def\blackdot{{\color{red}{\hskip-.0truecm\rule[-1mm]{4mm}{4mm}\hskip.2truecm}}\hskip-.3truecm}
\def\bluedot{{\color{blue}{\hskip-.0truecm\rule[-1mm]{4mm}{4mm}\hskip.2truecm}}\hskip-.3truecm}
\def\purpledot{{\color{colorpppp}{\hskip-.0truecm\rule[-1mm]{4mm}{4mm}\hskip.2truecm}}\hskip-.3truecm}
\def\greendot{{\color{colorgggg}{\hskip-.0truecm\rule[-1mm]{4mm}{4mm}\hskip.2truecm}}\hskip-.3truecm}
\def\cyandot{{\color{cyan}{\hskip-.0truecm\rule[-1mm]{4mm}{4mm}\hskip.2truecm}}\hskip-.3truecm}
\def\reddot{{\color{red}{\hskip-.0truecm\rule[-1mm]{4mm}{4mm}\hskip.2truecm}}\hskip-.3truecm}

\def\gdot{\greendot}
\def\bdot{\bluedot}
\def\ydot{\cyandot}
\def\rdot{\cyandot}

\def\fractext#1#2{{#1}/{#2}}
\def\fract#1#2{\frac{#1}{#2}}

\def\wojtek#1{\textcolor{blue}{#1}}
\def\igor#1{\textcolor{cyan}{\text{\bf #1}}}
\def\zach#1{\text{{\textcolor{colorqqqq}{#1}}}}
\def\igorf#1{\footnote{\text{{\textcolor{cyan}{\text{#1}}}}}}

\newcommand{\I}{\infty}
\newcommand{\AAA}{\mathrm{A}}
\newcommand{\BB}{\mathrm{B}}
\newcommand{\loc}{\mathrm{loc}} 
\newcommand{\uloc}{\mathrm{uloc}}
\newcommand{\MM}{M^{2,2}_{\mathcal C}}
\newcommand{\MMn}{M^{2,2}_{\mathcal C_n}}
\newcommand{\MMnq}{M^{2,q}_{\mathcal C_n}}
\newcommand{\MC}{\mathring M^{2,2}_{\mathcal C}}
\newcommand{\MCn}{\mathring M^{2,2}_{\mathcal C_n}}
\newcommand{\NN}{{\mathbb N}}

\newcommand{\eqnb}{\begin{equation}}
\newcommand{\eqne}{\end{equation}}
\newcommand\De{\Delta}
\newcommand\Om{\Omega}
\newcommand{\ga}{\gamma}
\newcommand{\la}{\lambda}
\newcommand\Ga{\Gamma}
\newcommand{\RR}{\mathbb{R}}
\newcommand{\CC}{\mathbb{C}}
\newcommand{\lec}{\lesssim}
\newcommand{\gec}{\gtrsim}
\newcommand{\na}{\nabla}
\newcommand{\cn}{{\mathcal{C}_n}}
\renewcommand{\d}{\,d}
\newcommand{\p}{\partial }
\newcommand{\ee}{{e}}

\begin{abstract} 
We prove existence of global-in-time weak solutions of the incompressible Navier-Stokes equations in the half-space $\RR^3_+$ with initial data in a weighted space that allow non-uniformly locally square integrable functions that grow at spatial infinity in an intermittent sense. The space for initial data is built on cubes whose sides $R$ are proportional to the distance to the origin and the square integral of the data is allowed to grow as a power of $R$.
 The existence is obtained via a new a priori estimate and stability result in the weighted space, as well as new pressure estimates. Also, we prove eventual regularity of such weak solutions, up to the boundary, for $(x,t)$ satisfying $t>c_1|x|^2 + c_2$, where $c_1,c_2>0$, for a large class of initial data $u_0$, with $c_1$ arbitrarily small. As an application of the existence theorem, we  construct global discretely self-similar solutions, thus extending the theory on the half-space to the same generality as the whole space. 

\end{abstract}

\maketitle

\setcounter{tocdepth}{2} 
\tableofcontents

\startnewsection{Introduction}{sec01}

We consider solutions to the three-dimensional incompressible Navier-Stokes equations (NSE)
  \begin{align}
   \begin{split}
   &\partial_t u-\Delta u +u\cdot\nabla u+\nabla p = 0,
   \\& 
   \nabla \cdot u=0,
   \end{split}
   \label{EQ01}
  \end{align}
posed on $\HH=\{(x_1,x_2,x_3):x_3>0\}$ satisfying homogeneous Dirichlet boundary condition on $\partial \HH\times (0,\I)$ and the initial condition
  \begin{equation}
   u(x,0)=u_0(x),
   \llabel{EQ02}
  \end{equation}
where $u_0 \in L^2_{\loc } (\overline{\HH})$ is divergence-free with $u_3=0$ on $\partial\HH$. If $u_0\in L^2$ then the existence of global-in-time weak solutions satisfying the strong energy inequality has been shown in the fundamental works of Leray \cite{leray} and Hopf \cite{hopf} (cf.~also \cite{CF,T,RRS,OP}), and are commonly referred to as \emph{Leray-Hopf weak solutions}. Such solutions enjoy additional structure, such as weak-strong uniqueness. One can also construct Leray-Hopf weak solutions that are  suitable (cf.~\cite{CKN,BIC}), that is that satisfy the local energy inequality (see \eqref{EQ155} below). This provides additional interior regularity, as shown in the celebrated work of Caffarelli, Kohn, and Nirenberg \cite{CKN} (see also \cite{RRS,O19}).  However, the questions of uniqueness and smoothness of such solutions remain open.

Recently, in \cite{LR}, the Leray theory was extended to uniformly locally square integrable data
in ${\mathbb R}^{3}$; cf.~also \cite{KiSe,KwTs} for some extensions. For global existence, these works assume some type of decay of the initial data as $|x|\to \infty$, either pointwise decay of a locally determined quantity, the decay of the $L^2$ norm confined to balls of unit radius, or the decay the oscillation computed over balls of unit radius. In \cite{LR-Morrey,BT8,BKT,FL,KwTs}, existence results are given in weighted spaces, which allow for lack of decay in some directions. The papers \cite{BKT,FL} additionally allow for growth in some directions.

The previously mentioned papers only concern $\RR^{3}$. On $\HH$, the question  of global existence is much more challenging and was only recently addressed for decaying uniformly locally  by Maekawa, Miura, and Prange \cite{MMP1}.  
Our goal is to establish the global existence of weak solutions in $\HH$
with data built on a dyadic tiling $\mathcal C$ of the half-space (see Figure 1 for an illustration projected on a half-plane), allowing
for growth in all directions. These data are on one hand general 
and on the other also well adapted for the study of self-similar solutions.

 \begin{figure}[h]
\centering
 \includegraphics[width=0.6\textwidth]{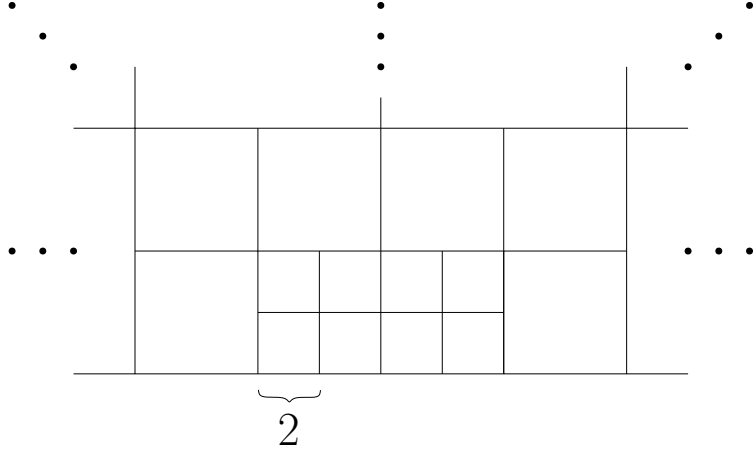}
 \nopagebreak
\captionsetup{width=0.8\textwidth}
  \captionof{figure}{The cover of $\HH$ by the collection $\mathcal{C}$. A scaled cover $\mathcal{C}_n$ is obtained by replacing $2$ with $2^n$. }\label{fig_cubes} 
\end{figure}

One of the main difficulties in the results on $\RR^3$ is the treatment of the pressure, which needs to be decomposed into the \emph{near field} and the \emph{far field}. Namely, given an open set $Q\subset \RR^3$ one can consider
\[\begin{split}
p_{\mathrm{near}} (x,t) +p_{\mathrm{far}} (x,t) - p_Q (t)&= -\frac13 | u(x,t) |^2 + \mathrm{p.v.}\int_{Q^*} K_{ij} (x-y) (u_i (y,t) u_j (y,t) ) \d y \\
&\hspace{2cm} + \int_{y\not \in Q^*} (K_{ij} (x-y)- K_{ij} (x_Q -y)) (u_i (y,t) u_j (y,t) ) \d y ,
\end{split}
\]
as in \cite{BK,BKT}, where $p_Q(t)$ is an arbitrary function of time, $K_{ij}(y)= \p_{ij} (4\pi |y|)^{-1}$, and $x_Q\in Q$ is fixed. In this context, the case of the half-space $\HH$ becomes much more difficult, as no such direct decomposition is available. In fact, except for the Helmholtz pressure, that is a solution of the nonhomogeneous Poisson equation 
\[
\begin{cases}
-\Delta p_{\Helm} &= \p_{ij} (u_i u_j ) \quad \text{ in } \HH,\\
\p_3 p_{\Helm} &= \mathrm{div}\, (u\,u_3) \quad \hspace{0.1cm}\text{ on } \p \HH,
\end{cases}
\]
one also needs to take into account the harmonic pressure, which is a solution of the Laplace equation with Neumann boundary condition,
\[
\begin{cases}
-\Delta p_{\harm} &=0 \hspace{1.8cm} \text{ in } \HH ,\\
\p_3 p_{\harm} &= \left. \Delta u_3  \right|_{x_3=0} \quad \text{ on } \p \HH,
\end{cases}
\]
where the boundary condition at $x_3=0$ should be understood in the sense of the trace of $\Delta u_3$. This part of the pressure function is absent in the case of the whole space $\RR^3$, and the Helmholtz pressure $p_{\Helm}$ admits much more sophisticated analysis than in the case of $\RR^3$. Both parts of the pressure become even more complicated when studying the corresponding near and far fields. In fact, only recently Maekawa, Miura, and Prange \cite{MMP1,MMP2} proved an analogue of the Lemari\'e-Rieusset theory in the half-space $\HH$, using explicit representation of the kernel for the Stokes equations in $\HH$, derived by Desch, Hieber, and Pr\"uss \cite{DHP}.

In this paper, we establish the global-in-time existence of weak solutions for intermittent initial data. To state our main results, we first define local energy weak solutions.
\begin{Definition}[Local energy solutions]\label{D02}
A
vector field $u\in L^2_{\loc}(\overline{\HH}\times [0,T))$, where $0<T<\infty$,
is a local energy solution to \eqref{EQ01} with divergence-free initial data $u_0\in L^2_{\loc}(\overline{\HH})$ if the following conditions hold:
\begin{enumerate}
\item $u\in \bigcap_{R>0} L^{\I}(0,T;L^2(B_R(0)\cap \HH ))$, $\nabla u \in L^2_{\loc}(\overline{\HH}\times [0,T])$ and $u|_{x_3=0}=0$,
\item for some $p\in \mathcal D'(\HH\times (0,T) )$, the pair $(u,p)$ is a distributional solution to \eqref{EQ01},
\item for all compact subsets $K$ of $\HH$ we have $u(t)\to u_0$ in $L^2(K)$ as $t\to 0_+$,
\item $u$ is suitable in the sense of Caffarelli-Kohn-Nirenberg, i.e., 
for all non-negative $\phi\in C_c^\infty (\overline{\HH}\times (0,T))$,
we have  the local energy inequality
\EQ{\label{EQ155}
&
2\iint |\nabla u|^2\phi\,dx\,dt 
\leq 
\iint |u|^2(\partial_t \phi + \Delta\phi )\,dx\,dt +\iint (|u|^2+2p)(u\cdot \nabla\phi)\,dx\,dt,
}
\item the function $t\mapsto \int u(x,t)\cdot w(x)\,dx$ is continuous on $[0,T)$ for any compactly supported $w\in L^2(\HH)$,
\item given a bounded, open set $\Om \subset \HH$, the pressure satisfies the \emph{local pressure expansion}, 
  \begin{equation}\begin{split}
   p = p_{\li,\loc} + p_{\li,\nonloc} + p_{\loc,\Helm}+ p_{\loc,\harm}+ p_{\nonloc,\Helm}+ p_{\harm, \leq 1}+ p_{\harm, \geq 1},
  \end{split}   
  \label{EQ07}
  \end{equation}
  up to a function of time, where the terms on the right-hand side are defined in \eqref{EQ20}, \eqref{EQ26}, \eqref{EQ32}, \eqref{EQ34}, and \eqref{EQ35} below, and estimated in \eqref{EQ37}.
\end{enumerate}
We say that $u$ is a local energy solution on $\HH\times [0,\I)$ if it is a local energy solution on $\HH\times [0,T)$ for all $T<\I$. 
\end{Definition}
This definition is primarily based on the one in \cite{KiSe}. Some works refer to this class (without part (6) and with minor modifications) as \textit{local Leray solutions} or \textit{Lemari\'e-Rieuesset type solutions}. This definition contains sufficient properties for work on regularity, see e.g.~\cite{CKN,L98,LS99,ESS,K,grujic,BS} among others, or in physical applications, see e.g.~\cite{DaGr}. We note that we do not assert any uniform control in $L^2_{\loc} $. The local pressure expansion in (6) above is inspired by the decomposition introduced by \cite{MMP1,MMP2}, and is unique up to a function of time; cf.~\eqref{EQ36}. 

Our main result is concerned with local energy weak solutions with initial $u_0$ belonging to a weighted space that allows growth of the kinetic energy at spatial infinity. 

To be precise, given $n\in{\mathbb N} $, we denote by $S_n$ the collection of 32 cubes of side-length $2^n$ that can be obtained by partitioning $\{x\in \HH: |x_i|\leq 2^{n+1} \text{~for~}i=1,2,3\}$. We let $R_k= \{x\in \HH: |x_i|< 2^k;i=1,2,3\}$ and we denote by $S_k$, for $k>n$, the collection of 28 cubes of side-length $2^k$ that can be obtained by partitioning $ \overline{R_{k+1}\setminus R_k}$. Also, set
  \eqnb\label{EQ_C}
   \mathcal{C}_n = \bigcup_{k\geq n} \bigcup_{Q\in S_k} Q
   ;
  \eqne
this is illustrated by Fig.~\ref{fig_cubes} with $2$ replaced by $2^n$. In other words, $S_k$ is the collection of cubes from $\mathcal{C}_n$ of side-length $2^{k}$. We set $\mathcal{C}= \mathcal{C}_1$.

\begin{Definition}\label{D01}
Given $p\in[1,\infty)$, $q\ge0$, and $n\geq 1$, we have $f\in M^{p,q}_{\mathcal{C}_n}$ if  
    \begin{equation}
     \|f\|_{M^{p,q}_{\mathcal{C}_n}}^p =\sup_{Q\in \mathcal{C}_n} \frac 1 {|Q|^{\fract{q}{3}}}\int_{Q} |f(x)|^p\,\d x<\infty.
    \llabel{EQ03}    
  \end{equation}
We denote by $\mathring M^{p,q}_{\mathcal C}$ the closure 
in $M^{p,q}_{\mathcal C}$ of divergence-free, smooth functions, which are compactly supported in $\HH$.
\end{Definition} 
We note that 
 \begin{equation}
   \frac 1{|Q|^{\fract{q}{3}}} \int_Q |f|^p \,dx \to 0 \text{ as } |Q|\to \infty, Q\in \mathcal C,
    \label{EQ06}
  \end{equation}
for $f\in \mathring M^{p,q}_{\mathcal C}$, which was shown in \cite{BKT}. 
 
In the context of the whole space $\RR^3$, the spaces $M^{p,q}_{\mathcal{C}}$ are discussed in detail in \cite{BKT} and from the perspective of interpolation theory in \cite{FL-pressure}. The same observations apply here; in particular, the choice of tiling does not matter so long as elements have length scales comparable to their distance from the origin and dyadic cubes can be replaced by balls centered at the origin as in \cite{Basson}. These spaces can be characterized as inhomogeneous Herz spaces \cite{Tsu}. They were introduced to the study of fluids by Basson \cite{Basson} to establish local existence for the Navier-Stokes equations in two dimensions. In comparison to \cite{BKT}, here we build the divergence free condition and the decay condition into the space $\MC$ while in \cite{BKT} we only built the decay condition into this space. 

Our main theorem is concerned with the global-in-time existence with initial data in $\MC$.
\cole
\begin{Theorem}[Global existence of local energy solutions]\label{T01}
Given $u_0\in  \MC$, there exists a local energy solution
$u$ on $\HH\times (0,\I)$  with the initial data $u_0$.
\end{Theorem}
\colb 
In comparison with the $L^2_{\uloc}(\HH)$ setting of \cite{MMP1,MMP2}, we note that neither of the two spaces $L^2_{\uloc}(\HH)$ and $\MC$ is contained in the other. For example, if $u_0$ is a constant (or periodic) function, then $u_0 \in  L^2_{\uloc}(\HH) \setminus \MC$, while
\[
 \sum_{k\geq 1} 2^{qk/2} \chi_{B_1 (2^k e_3) } \in M^{2,q}_{\mathcal{C}} \setminus L^2_{\uloc } (\HH),
\] 
cf.~\cite{BK}. In this context, Theorem~\ref{T01} is the first result asserting global-in-time existence of weak solutions in the half-space that allows intermittent initial data. 

The proof of Theorem~\ref{T01} is based on new a~priori estimates and a stability result. To be more precise, setting
  \begin{equation}
  \alpha_n (t) =\sup_{s\in[0,t]}\| u (s) \|_{M_{\mathcal{C}_n}^{2,q}}^2
  \quad \text{ and } \quad 
  \beta_n(t) = \sup_{Q\in \mathcal{C}_n} \frac{1}{|Q|^{\fract{q}{3}}} \int_0^t \int_{Q} |\na u |^2
   \label{EQ08}  
  \end{equation}
  for $n\in {\mathbb N}$ and $q\in (0,2] $, we have the following a~priori estimate.
\cole
\begin{Theorem}[A priori bound]\label{T02}
Assume that $1\leq q\leq 2$. There exists $\gamma_0\geq1$ with the following property. 
Let $n\in{\mathbb N}$ and $u_0\in \mathring M^{2,q}_{\mathcal C_n}$, 
and suppose that $(u,p)$ is a local energy solution on $\HH\times (0,\infty)$ with the initial data $u_0$ such that    
  \EQ{
   \esssup_{0<s<t}(\alpha_n(s)+\beta_n(s))<\infty,
   \llabel{EQ09}
  }
for all $t<\infty$. Then there exists a constant $C\geq 1$ such that 
  \EQ{\label{EQ11}
  \alpha_n(T_n)  + \beta_n(T_n)
      \leq 
       C \alpha(0)
  }
for some $T_n >0 $, such that 
  \begin{equation}
   T_n
   \geq
   \min\left\lbrace
          2^{n},
            \| u_0 \|_{M^{2,q}_{\cn}}^{-1/\gamma_0}
       \right\rbrace 
   .
   \label{EQ10}
  \end{equation}
\end{Theorem}
\colb
Another important ingredient in the proof of Theorem~\ref{T01} is the following stability result.

\cole
\begin{Theorem}\label{T03} Let $u_0\in \MC(\HH)$ and suppose that $\{u_0^{(k)}\}_{k\geq 1} \subset \MC$ is such that $\| u_0^{(k)}- u_0 \|_{M^{2,2}_\mathcal{C} }\to 0$. Moreover, suppose that $\{  (u^{(k)},p^{(k)})\}_{k\geq 1}$ is a collection of local energy solutions with initial data $u_0^{(k)}$ that satisfy the a~priori estimate of Theorem~\ref{T02}. Then there exists a subsequence $\{ k_l \}_{l\geq 1}$ such that $(u^{(k_l)},p^{(k_l)})$ converges in a weak sense to $(u,p)$ that is a global-in-time local energy solution with initial data $u_0$. In addition, this solution satisfies the a~priori estimate \eqref{EQ11}, and, given $\Om \Subset \HH$, each part of the local pressure expansion of $p^{(k_l)}$ converges strongly in $L^{\frac32}(\Om \times (0,T))$ to the corresponding part of the local pressure expansion of $p$, for every $T>0$.
\end{Theorem}
\colb

The most difficult part of the proofs of the above two theorems is the treatment of the pressure function. We note that each of the pressure parts $p_{\li,\loc}$, $p_{\li,\nonloc}$, $p_{\loc,\Helm}$, $p_{\loc,\harm}$, $p_{\nonloc,\Helm}$, $p_{\harm, \leq 1}$, $p_{\harm, \geq 1}$ is defined in a similar way as in the works \cite{MMP1,MMP2} in the uniformly locally integrable setting. However there are important differences in our treatment of each of these parts in our estimates (see \eqref{EQ37}). Our pressure bounds allow us to obtain the a~priori estimate \eqref{EQ11} as well as the strong convergence under the perturbation of the initial data $u_0$ in $M^{2,2}_\mathcal{C}$ mentioned in Theorem~\ref{T03}. Consequently, we obtain the explicit representation \eqref{EQ07} of the pressure for the weak solutions constructed in Theorem~\ref{T01}.

Our stability result can also be applied to construct self-similar and discretely self-similar solutions with very rough data. Recall that if $u$ solves \eqref{EQ01}, then so does $u^{(\la)}(x,t):=\la u (\la x,\la^2t)$ for $\la>0$. Self-similar (SS) solutions, i.e., solutions invariant with respect to the scaling of \eqref{EQ01} for all scaling factors $\la>0$, are noteworthy candidates for the non-uniqueness of Leray-Hopf weak solutions, as demonstrated by  \cite{JiaSveral-nonuniqueness,GuiSve}. On the other hand, discretely self-similar (DSS) solutions, i.e.,~solutions that satisfy the scaling invariance possibly only for some $\la>1$, are candidates for the failure of  eventual regularity \cite{BT1,BT8}. For small data, the existence and uniqueness of such solutions follow easily from the classical well-posedness results; see \cite{Koch-Tataru} and the references therein. The  more interesting case of large data has been only recently solved by Jia and \v{S}ver\'{a}k \cite{JiaSverak}, and some improvements and new approaches have been developed by e.g.~\cite{Tsai-DSSI,KT-SSHS,BT1,BT5,AlBr,CW,FL}. The roughest class of scaling invariant initial data for which existence is known is $L^2_\loc (\RR^3)$ \cite{CW,BT5,LR2,FL}. Note that if $u_0\in L^2_\loc(\RR^3)$ is scaling invariant then it belongs to the $\RR^3$ version of the space $M^{2,1}_{\mathcal C}$, see \cite{BK}. Indeed, this observation led the first and second authors to study these spaces in \cite{BK}.  

In the case of the half-space, Tsai, and Korobkov\cite{KT-SSHS} established the original theory of Jia and  \v{S}ver\'{a}k \cite{JiaSverak} for smooth, self-similar data via a new method and, later, Tsai and the first author \cite{BT2} addressed rough, discretely self-similar initial data in $L^{3,\I}$ with arbitrary scaling factor.  As a consequence of Theorem~\ref{T03}, we prove that any SS/DSS initial data in $\MC$ gives rise to a SS/DSS solution. This class of initial data corresponds to the roughest case for $\RR^3$ \cite{CW,BT5,LR2,FL} with a suitable boundary condition imposed. {To see why this is true, assume that $u_0\in L^2_\loc (\HH)$ is divergence free with vanishing normal component at the boundary (this is the boundary condition implicit in the space $\MC$ since $\MC$ is obtained by taking the closure of compactly supported test functions; see \cite[Proposition 1.5]{CF}). If, additionally, $u_0$ is scaling invariant, then, by a re-scaling argument, $u_0\in M^{2,1}_\mathcal C \subset M^{2,2}_{\mathcal C}$. Furthermore, we have  $u_0\in \MC$ because $u_0$ decays at spatial infinity (from membership in $M^{2,1}_\mathcal C$) and satisfies the correct boundary condition.}

\begin{Theorem}[Global existence of self-similar solutions]\label{T06}
Assume $u_0\in \MC$ is divergence free  and is discretely self-similar for some $\la>1$. Then there exists a global-in-time local energy solution $u$ with data $u_0$ that is discretely self-similar. If $u_0$ is self-similar, then so is $u$.
\end{Theorem}

The proof of Theorem~\ref{T06} uses our new a priori bounds to construct discretely self-similar solutions  via the stability result of Theorem~\ref{T03} applied to a sequence of solutions given by \cite{BT2}. This is similar to the approach taken by the first author and Tsai in \cite{BT5} in the case of $\RR^3$. However, an important difference is that \cite{BT5} deals with the non-local pressure by exploiting the DSS scaling to localize the far-field part, while in the present paper this technical step is unnecessary  because the far-field part of the pressure is controlled using the weighted $L^2$ framework.\\

Finally,  we show that local energy solutions eventually become regular, up to the boundary, provided $u_0$ belongs to a subspace of $\mathring M^{2,2}_{\mathcal{C}}$. This provides an extension of Theorem D of \cite{CKN} in the setting of half-space $\HH$ by allowing growth of $u_0$ at spatial infinity. We show that if $u_0$ is in  $\mathring M_{{\mathcal C}}^{2,q}$ where $q<1$, then an ensuing local energy solution has eventual regularity above a parabola. 
\cole
\begin{Theorem}
\label{T04}
If $u_0\in \mathring M_{{\mathcal C}}^{2,q}$  for some $q\in (0,1)$ and $\epsilon_0\in(0,1)$, then there exists $M>0$ such that any local energy solution $u$ with initial data $u_0$ is regular (up to the boundary) in the region
  \begin{equation}
   \{(x,t)\in \HH\times(0,\infty):
     t\geq  \epsilon_0|x|^2 + M
   \}
   .
   \llabel{EQ12}
  \end{equation}
\end{Theorem}
\colb

The paper is organized as follows. Section \ref{sec03} contains the study of the local pressure expansion and provides the main estimates \eqref{EQ37} for all pressure parts. This is then used in Section \ref{sec04}, where we prove the a~priori estimate, Theorem~\ref{T02}. Section \ref{sec05} contains the proof of the stability result, Theorem~\ref{T03}. We then prove the main existence results, Theorem~\ref{T01} and Theorem~\ref{T06} in Section \ref{sec06}. The proof of the eventual regularity result, Theorem~\ref{T04} is provided in Section~\ref{sec07}.

\section{Pressure formula and estimates}\label{sec03}

 Given a bounded, open set $\Om\subset \HH$ and $N$ let $\underline{n},\overline{n} \geq N$ be such that 
\eqnb\label{EQ44}
\Om \text{ can be covered using cubes from } S_{\underline{n}}\cup S_{\underline{n}+1} \cup \cdots \cup S_{\overline{n}}.\eqne
 Note that if $\Om \in \mathcal{C}_N$ for some $N\geq 1$, then $\un = \on =n$, where $n$ is such that $\Om \in S_n$. (Recall \eqref{EQ_C} for the definition of the family $\mathcal{C}_N$, see also Fig. \ref{fig_cubes}.)

Given $\Om$ and $N$, let $Q$ be the union of (closed) cubes from $\mathcal{C}_N$ that have nonempty intersection with $\Om$. Denote by $Q^*$ the union of the neighbors of $Q$, i.e., the union of $Q$ and all cubes from $\mathcal{C}_N$ that share at least one common boundary point with $Q$. We similarly define $Q^{**}$ and $Q^{***}$. We set $\chi\in C_0^\infty (\overline{\HH},[0,1])$ be such that $\chi=1$ on a neighborhood of $Q$ and $\chi=0$ outside $Q^*$, and we define $\chi_*$ and $\chi_{**}$ analogously. In the pressure estimates below we shall use the following simple geometric fact: If $\xi \in Q$ and  $z\in \{ \chi <1 \}$ is such that $z\in \tilde{Q} $  for some $\tilde{Q} \in S^k\subset \mathcal{C}_N $, then
\eqnb\label{EQ14}
|\xi'-z' | + \xi_3 +z_3 
\gec
\begin{cases} 
\,2^{\un} \qquad &k\leq \un ,\\
\, 2^k &k\geq \un+1,
\end{cases}
\eqne
where we used the notation $x=(x',x_3)$ to distinguish the horizontal component $x'$ and the vertical component $x_3$ of any given point $x$ of the half-space $\HH$. 
Indeed, if $k\leq  \un$, then either $|\xi'-z'| \gtrsim 2^{\un }$ (if $Q$ touches the plane $\p \HH$ and $z$ does not lie in a cube above $Q$), $z_3\gtrsim 2^{\un } $ (if $z$ does lie in a cube above $Q$) or $\xi_3\gtrsim 2^{\un }$ (if $Q$ does not touch the plane). The case $k\geq \un +1$ follows similarly as either $|\xi'-z'| \gtrsim 2^k $ (if $z$ lies in a cube touching the plane $\p \HH$ and $\xi$ does not lie in a cube above it), $\xi_3\gtrsim 2^k$ (if it does) or $z_3 \gtrsim 2^k$ (if $z$ lies in a cube not touching the plane). Furthermore, note that
\eqnb\label{EQ15}
|\xi - z| \lec 2^{\on} \qquad \text{ for any }\xi,z\in Q^{***}
.
\eqne
We write
\begin{equation}
u=u_{\li}+u_{\loc}+u_{\nonloc},\quad p=p_{\li}+p_{\loc}+p_{\nonloc},
    \llabel{EQ16}
  \end{equation}
where the terms on the right-hand sides are solutions to 
the linear part
\begin{equation}
\begin{cases}
\p_t u_{\li} - \Delta u_{\li} + \na p_{\li} =0,\\
\na \cdot u_{\li} =0,\\
u_{\li}|_{\{ z_3=0 \} } = 0,\\
u_{\li}|_{\{ t=0 \} } = u_0,
\end{cases}
    \llabel{EQ17}
  \end{equation}
the local part
\eqnb\label{EQ18}
\begin{cases}
\p_t u_{\loc} - \Delta u_{\loc} + \na p_{\loc} =-\na \cdot (\chi_{**} u\otimes u), \\
\na \cdot u_{\loc} =0,\\
u_{\loc}|_{\{ z_3=0 \} } = 0,\\
u_{\loc}|_{\{ t=0 \} } = 0,
\end{cases}
\eqne
and the nonlocal part
\begin{equation}
\begin{cases}
\p_t u_{\nonloc} - \Delta u_{\nonloc} + \na p_{\nonloc} =-\na \cdot ((1-\chi_{**}) u\otimes u), \\
\na \cdot u_{\nonloc} =0,\\
u_{\nonloc}|_{\{ z_3=0 \} } = 0,\\
u_{\nonloc}|_{\{ t=0 \} } = 0.
\end{cases}
    \llabel{EQ19}
  \end{equation}
We note that each of the pressure components enters the equation with a gradient, and thus it can be modified by an arbitrary function of $t$. 
We have the representation
  \begin{equation}\label{EQ20}\begin{split}
    p_{\li}(x,t)& = \frac{1}{2\pi i} \int_{\Gamma } \ee^{\la t} \int_{\HH} 
        \left(  
          \chi q_\la (x'-z',x_3,z_3) \cdot u_0' (z) 
          +(1-\chi) q_{\la ,x,x_Q } (z ) \cdot u_0' (z) \right) \d z' \, \d z_3 \, \d \la \\
&= p_{\li,\loc} (x,t) + p_{\li,\nonloc} (x,t),
    \end{split}
  \end{equation}
where $x=(x',x_3)$ and
 $\Gamma = \{  \la \in \CC : | \mathrm{arg}\la | = \eta \}\cap \{ \la \in \CC \colon | \mathrm{arg}\la | \leq  \eta , |\la |=\kappa \}$ with $\eta \in (\pi/2 , \pi)$ and $\kappa \in (0,1)$,
\begin{equation}
q_\lambda (x',x_3, z_3 ) \coloneqq i \int_{\RR^2} \ee^{ix'\cdot \xi} \ee^{-|\xi | x_d } \ee^{-\omega_\la (\xi ) z_3} \left( \frac{\xi}{|\xi |} + \frac{\xi}{\omega_\la (\xi )} \right) \d \xi,
    \llabel{EQ21}
  \end{equation}
$\omega_\la (\xi )\coloneqq \sqrt{\la + |\xi |^2}$, and 
\eqnb\llabel{EQ22}
q_{\la , x ,x_\Om } (z) \coloneqq q_\lambda (x'-z',x_3, z_3 ) - q_\la (x_{\Om}'-z' ,x_{\Om,3},z_3),
\eqne
cf.~(2.5)--(2.9) in \cite{MMP1} and (2.8e) in \cite{MMP2}.
Above, $x_{\Om}=(x_{\Om}',x_{\Om,3})$ stands for any fixed point of $\Om$; if $\Om$ is a cube, we let $x_{\Om}$ be the center of~$\Om$.
Moreover, we have the pointwise estimates
  \eqnb\label{EQ23}
  \begin{split}
  | \na_x^m q_\la (x'-z',x_3,z_3) | 
     &\lec_m  
       \frac{\ee^{-|\la |^{\fract{1}{2}}z_3} }{(|x'-z'| + x_3 +z_3 )^{2+m}}
  \comma m\in {\mathbb N}_0
  ,
  \end{split}
  \eqne
which are proven in Proposition 3.7 in \cite{MMP2}.

As for the local (nonlinear) pressure $p_{\loc}$, we use the Helmholtz decomposition in the half-space to write
  \begin{equation}\label{EQ24}
   \na \cdot (\chi_{**} u\otimes u)=  \mathbb{P}\na \cdot (\chi_{**} u\otimes u) + \na p_{\loc,\Helm}
    ,
  \end{equation}
where $p_{\loc,\Helm}$ is the solution to Poisson equation with the Neumann boundary condition
\begin{equation}
\begin{cases}
- \Delta p_{\loc,\Helm} = \p_i \p_j (\chi_{**} u_iu_j) \quad \text{ in } \HH,\\
\p_3 p_{\loc,\Helm} = \p_i (\chi_{**} u_i u_3)  \quad \text{ on } \p \HH.
\end{cases}
    \llabel{EQ25}
  \end{equation}
The solution is given by 
\begin{equation}\label{EQ26}
p_{\loc,\Helm} (x,t) = c_0 \chi_{**} |u|^2 (x,t) +  \int_{\HH} \p_{z_i} \p_{z_j} N (x,z) \chi_{**}(z) u_i (z,t) u_j (z,t) \d z
   ,
  \end{equation}
where  $c_0$ is a constant; also
\eqnb\label{EQ27}
N(x,z)= \frac{1}{4\pi} \left(\frac{1}{|x-z|}+\frac{1}{|\overline{x}-z|}\right)
\eqne denotes the Neumann kernel for the half-space, where
$\overline{x}= (x_1,x_2,-x_3)$ is the reflection of $x$ with respect to the boundary $\p \HH$. With this definition of the local Helmholtz pressure $p_{\loc,\Helm}$ one can use Fourier analytic methods (cf.~(6.2) in \cite{MMP2} and Appendix A.1 in \cite{MMP1}) to deduce that for the Leray projection, we have
\begin{equation}
  \mathbb{P}\na \cdot (\chi_{**} u\otimes u) = F_A + F_B,
    \llabel{EQ28}
  \end{equation}
where $F_A$ is a 2D vector function whose components are finite sums of the terms of the form
  \eqnb\label{EQ29}
  \p_j \left( \chi_{**} u_k u_l \right)
  \eqne
where $j,k,l\in \{ 1,2,3 \}$, and $F_B(z,s)=F_B(z',z_3,s)$ is a finite sum of 2D~vectors of the form
  \eqnb\label{EQ30}
  m(D') \na' \otimes \na' \int_0^\infty  \left( \left( P(\cdot,|z_3 -y_3|) + P(\cdot,z_3+y_3) \right)\ast (\chi_{**} v\otimes w ) (y_3) \right)(z',s) \d y_3  ,
  \eqne
where $v$ and $w$ denotes various 2D vectors whose components are chosen among $u_1$, $u_2$, or $u_3$;
also,  $m(D')$ denotes a multiplier in horizontal variable $z'$ that is homogeneous of degree $0$, and $\widehat{P}(\xi', t) =\ee^{-t|\xi' |}$,
i.e., $P$ is the 2D Poisson kernel.
The convolution in \eqref{EQ30} is in the horizontal variables.
Thus, letting $(u_{\loc,\harm},p_{\loc,\harm})$ be a solution to \eqref{EQ18} but with the right-hand side replaced by $- \mathbb{P}\na \cdot (\chi_{**} u\otimes u)$ we see that
  \begin{equation}
    p_{\loc} = p_{\loc,\Helm}+p_{\loc,\harm} 
    ,
   \llabel{EQ31}
  \end{equation}
by applying the Helmholtz decomposition \eqref{EQ24} to \eqref{EQ18},
and, from the Duhamel principle
  \eqnb\label{EQ32}
    p_{\loc,\harm} (x,t) 
          =\frac{1}{2\pi i }\int_0^t \int_\Gamma \ee^{(t-s)\la } \int_{\HH} q_{\la} (x'-z',x_3,z_3 ) \cdot ( F_A(z,s)+F_B(z,s) ) \d z \, \d \la \,\d s
   .
  \eqne

As for the nonlocal (nonlinear) pressure $p_{\nonloc}$, we use the Helmholtz decomposition to write (similarly as in the case of $p_{\loc}$) 
\begin{equation}
\na \cdot ((1-\chi_{**}) u\otimes u)= \mathbb{P}\na \cdot ((1-\chi_{**}) u\otimes u) + \na p_{\nonloc,\Helm},
    \llabel{EQ33}
  \end{equation}
where
\begin{equation}\label{EQ34}
p_{\nonloc,\Helm} (x,t) = \int_{\HH} \p_{z_i} \p_{z_j} N_{x,x_{\Om}} (z)(1- \chi_{**}(z)) u_i (z,t) u_j (z,t) \d z
  \end{equation}
and $N_{x,x_{\Om}} (z)= N(x,z)-N(x_{\Om},z)$. Note that by introducing $N (x_{\Om} , z )$ we have modified $p_{\nonloc,\Helm}$ by a function of $t$ only (cf.~local Helmholtz pressure, $p_{\loc,\Helm}$, above), which therefore makes no change to $\na p_{\nonloc,\Helm}$.

Similarly, we can modify the nonlocal harmonic pressure by writing 
\eqnb\label{EQ35}\begin{split}
p_{\harm}(x,t)&= \frac{1}{2\pi i} \int_0^t \int_\Gamma \ee^{\la (t-s)} \int_{\HH} q_\la (x'-z', x_3 ,z_3) \cdot \chi_* F_B(z,s) \d z'\,\d z_3 \, \d \la \, \d s \\
&\hspace{1cm}+ \frac{1}{2\pi i} \int_0^t \int_\Gamma \ee^{\la (t-s)} \int_{\HH} q_{\la, x,x_{\Om }} (z) \cdot (1-\chi_*)( F_A(z,s)+F_B(z,s) ) \d z \, \d \la \, \d s\\
&= p_{\harm, \leq 1}(x,t)+p_{\harm,\geq 1}(x,t)
\end{split}
\eqne
(see also (2.17) in \cite{MMP1}), 
where $F_A$ and $F_B$ are defined as in \eqref{EQ29} and \eqref{EQ30} with $\chi_{**}$ replaced by $(1-\chi_{**})$. Note that there is no $F_A$ part in $p_{\harm,\leq 1}$ as $\chi_*$ vanishes on $\supp {(1-\chi_{**})}$.\\

We point out that the above representation of the pressure function is \emph{unique} on a given bounded open set $\Omega \subset \HH$, up to a function of time, of any local energy solution $u$. Indeed, each of the pressure parts $p_{\li,\loc}$, $p_{\li,\nonloc}$, $p_{\loc,\Helm}$, $p_{\loc,\harm}$, $p_{\nonloc,\Helm}$, $p_{\harm, \leq 1}$, $p_{\harm, \geq 1}$ depends only on $u$ (rather than on our decomposition $u_{\li}+u_{\loc }+u_{\nonloc}$), and so uniqueness (up to a function of time) follows from the distributional form of the Navier-Stokes equations (recall Definition \ref{D02}). In other words, if given ${\Om}, \Om' \subset \HH$ are such that $\Om\subset \Om'$ and if we define the pressure functions $p_{\Om}, p_{\Om'}$ as the sum of the above pressure parts (respectively), then
\eqnb\label{EQ36}
p_{\Om } - p_{\Om'} = c_{\Om , \Om'} (t)
\eqne
on $\Om$ for some $c_{\Om , \Om'}$, a function of time only.\\

In the remaining part of this section, we fix $N\geq 1$ and we prove the following estimates on any given bounded open set $\Om \subset \HH$.
  \eqnb\label{EQ37}
  \begin{split}
  \| p_{\li,\loc} (t) \|_{L^2 (\Om )}&\lec  2^{\frac{q+1}2\on }  \|u_0  \|_{M^{2,q}}  t^{-\frac{3}{4}},\\
  \| p_{\li,\nonloc} (t) \|_{L^\infty (\Om )}&\lec  2^{\on + \fract{q-4}{2}\un} \| u_0 \|_{M^{2,q}}  t^{-\frac{3}{4}}\\ 
  \| p_{\loc,\Helm} (t) \|_{L^{\fract{3}{2}}(\HH)} &\lec \| u (t) \|^{2}_{L^3 (Q^{***})} ,\\
  \| p_{\loc,\harm} -\theta\|_{L^{\frac32} ((0,t); L^{\frac{17}{10}} (\HH))} 
          &\lec   2^{q\on } \left(
       \| \alpha \|_{L^{\fract{39}{5}}(0,t)}^{\fract{13}{34}} \beta(t)^{\fract{21}{34}}+ 2^{-\fract{4}{17}\un}\| \alpha \|_{L^3 (0,t)}^{\fract12}   \beta(t)^{\fract12}+2^{-\fract{21}{17}\un} \| \alpha \|_{L^{\fract32}(0,t)}   \right)\\
  \| p_{\nonloc,\Helm} (t) \|_{L^{\infty}(\Om  )} &\lec  2^{\on + (q-4)\un } \|u(t) \|_{M^{2,q}}^2  ,\\
   \| p_{\harm, \geq 1} \|_{L^r ((0,T);L^\infty (\Om ))} &\lec 2^{\on +(q-4)\un  } (1+T)^{\gamma} \left( 2^{\frac{4\delta}3\un }\| \alpha \|_{L^{\frac{r(1-\delta )}{1-r\delta }}(0,T)}^{1-\delta }  \beta(T)^{\delta }\right.\\
   &\hspace{5cm}\left. +(1+T^{\frac12} 2^{\on q -(1+q)\un} ) \| \alpha \|_{L^r (0,T)}  \right) \\
  \| p_{\harm, \leq 1} (t)\|_{L^\infty (\Om )} &\lec 2^{\frac{\on}{4} + \left(q-4\right)\un} 
        \alpha(t)  t^{\frac38}
  ,
  \end{split}
  \eqne
for $t>0$, $r\in [1,\infty )$, $\gamma \in (0,1)$, $q\in (0,3)$, and $\delta \in (0, \min\{1/r,3/4,3q/2)\}$, where $\theta$ is a function of $t$ only. The implicit constants  depend on $q$, $r$ ,$ \delta$, $\gamma$, $\kappa$ and $\eta$; we used the notation \eqref{EQ08}.
For simplicity of notation, we abbreviate
$M^{2,q}= M^{2,q}_{\mathcal{C}_N}$.

Recall \eqref{EQ44} that $\Omega$ also gives $\un,\on \geq N$. We use these two indices to obtain the local pressure expansion (recall Definition \ref{D02}(6)) for \emph{any bounded open set $\Om $}. Since the cutoff functions $\chi$, $\chi_*$, $\chi_{**}$ appearing in the systems above are adapted to the geometry of the family $\mathcal{C}_N$, the same is true of the above estimates. It is therefore natural to quantify to what extend is any given  $\Om $ adapted to $C_N$, which we achieve using the two indices $\un, \on $. 

We note that the implicit constants in \eqref{EQ37} do not depend on $N$. If $\Om \in \mathcal{C}_N$ for some $N\geq 1$, then we have $\underline{n}=\overline{n}=n$, where $n\geq 1$ is such that $2^n$ is the side-length of $Q$. For such $Q$ the estimates reduce by replacing $2^{\on}$ and $2^{\un}$ with $|Q|^{1/3}$.
 An important property to keep in mind is that if $\Omega$ is a cube from $ \mathcal{C}_1$ that also belongs to $\mathcal{C}_N$ for some $N>1$ then the estimates get sharper for larger $N$. 

\cole
\begin{Lemma}[Estimate for $p_{\li,\loc}$]\label{L01}
For every $t>0$, we have
  \begin{equation}
   \| p_{\li,\loc} (t) \|_{L^2 (\Om )}\lec 2^{\on /2}   \|u_0  \|_{L^2(Q^{*})}  t^{-\frac{3}{4}}
   .
   \llabel{EQ38}
  \end{equation}
\end{Lemma}
\colb

\begin{proof}[Proof of Lemma~\ref{L01}]
Fix $t>0$, and note that $\Om \subset \Om' \times \Omega_3$, where $\Omega'$ denotes the projection of $\Omega$ onto the $(x_1,x_2)$ plane, and $\Om_3$ onto the $Ox_3$ axis. For $x_3\in \Om_3$, we use \eqref{EQ20} to get 
  \eqnb\label{EQ39}
  \begin{split}
  \| p_{\li,\loc} (\cdot ,x_3) \|_{L^2_{x'} (\Om')} &\lec \int_\Ga \ee^{t\re\la }  
       \left\|  \int_{\HH}  
             \chi  q_\la (x' -z',x_3,z_3)\cdot u_0 (z)  \d z' \, \d z_3  
       \right\|_{L^2_{x'} (\Om')} \d |\la| \\
  & \lec \int_\Ga \ee^{t\re\la } \int_0^\infty \ee^{-|\la |^{\fract{1}{2}}z_3} \left\|\int_{\RR^2}  (|x' -z'|+x_3+z_3)^{-2} |\chi u_0 (z) |  \d z' \right\|_{L^2_{x'} (\Om')}  \d z_3 \d |\la |
  .
  \end{split}
  \eqne
Since 
  \begin{equation}
  |x'-z'|\leq |x'|+|z'| \lec 2^{\on}  ,
    \label{EQ40}  
  \end{equation}
for every $x\in \Om$ and $z\in Q^{*}$, it follows that for every $z_3$ 
  \eqnb
  \label{EQ41}
  \begin{split}
  &\left\|\int_{\RR^2}  (|x'-z'|+x_3+z_3)^{-2} |\chi u_0 (z) |  \d z' \right\|_{L^2_{x'} (\Om')}  
  \indeq\\&\indeq
  = \left\|\int_{\RR^2} 1_{C2^{\on} }(x' -z') (|x' -z'|+x_3+z_3)^{-2} |\chi u_0 (z) |  \d z' \right\|_{L^2_{x'} (\Om')} \\
  &\indeq \lec \| \chi u_0 (z' , z_3 ) \|_{L_{z'}^2(\RR^2)} \int_{|y'|\leq C 2^{\on }} (|y'|+x_3+z_3)^{-2} \d y'\\
   &\indeq\lec \| \chi u_0 (z' , z_3 ) \|_{L_{z'}^2(\RR^2)} 2^{\on /4}(x_3+z_3)^{-\fract{1}{4}}
   ,
  \end{split}
  \eqne
where we used  
$\int_{|y'|\leq a } (|y'|+b )^{-2} \d y' = 2\pi  ( \fractext{-a}{(a+b)} + \log (1+a/b))\leq c_\alpha (a/b)^\alpha$ for any $\alpha \in (0,1)$.
Therefore,
  \begin{align}
  \begin{split}
  &\left\|\int_0^\infty \ee^{-|\la |^{\fract{1}{2}}z_3} \left\|\int_{\RR^2}  (|x' -z'|+x_3+z_3)^{-2} |\chi u_0 (z) |  \d z' \right\|_{L^2_{x'} (\Om')}  \d z_3  \right\|_{L^2_{x_3} (\Om_3)} 
  \\&\indeq
  \lec 2^{\on /4} \| x_3^{-\fract{1}{4}} \|_{L_{x_3}^2 (\Om_3)}   
    \int_0^\infty \ee^{-|\la |^{\fract{1}{2}}z_3} \| \chi u_0 (z' , z_3 ) \|_{L^2_{z'}(\RR^2)}   \d z_3   
  \\&\indeq
  \lec 2^{\on /2}  \left( \int_0^\infty \ee^{-2|\la |^{\fract{1}{2}}z_3} \d z_3 \right)^{\fractext12}  \| \chi u_0 \|_{L^2(\HH)}  
  \\&\indeq
  \lec  2^{\on /2}  \| \chi u_0  \|_{L^2(\HH)}   {|\la |^{-\fract{1}{4}}}   
  .
  \end{split}
  \label{EQ42}  
  \end{align}
Finally, including the integral in $\la$,  we obtain from \eqref{EQ39},
  \begin{equation}
  \| p_{\li,\loc} \|_{L^2 (\Om )}
      \lec  2^{\on /2}  \| \chi u_0  \|_{L^2(\HH)}   \int_\Gamma \ee^{t\re \la }{|\la |^{-\fract{1}{4}}}   \d |\la | \lec 2^{\on /2}  \| \chi u_0  \|_{L^2(\HH)}  t^{-\frac{3}{4}}
  ,
    \llabel{EQ43}  
  \end{equation}
and the proof is concluded.
\end{proof}

\cole
\begin{Lemma}[Estimate for $p_{\li,\nonloc}$]
\label{L02}
For every $t>0$ 
  \begin{equation}
  \| p_{\li,\nonloc} (t) \|_{L^\infty (\Om)}\lec_{q }2^{\on + \fract{q-4}{2}\un} \| u_0 \|_{M^{2,q}}  t^{-\frac{3}{4}}
  ,
   \llabel{EQ45}  
  \end{equation}
where $q\in (0,3)$.
\end{Lemma}
\colb

\begin{proof}[Proof of Lemma~\ref{L02}]
For every $x\in \Om$ we use \eqref{EQ23}, as well as the geometric properties \eqref{EQ14}, \eqref{EQ15}, to obtain
 \begin{equation}
   \begin{split}
\left| \int_{\HH}  q_{\la ,x,x_{\Om} } (z ) (1-\chi ) u_0 (z)  \d z \right|   &\lec 2^{\on}   \int_0^\infty \int_{\RR^2} \frac{\ee^{-|\la |^{\fract{1}{2}} z_3} }{(|\xi'-z'| + \xi_3 +z_3 )^{3}}  |(1-\chi ) u_0 (z)|  \d z' \,\d z_3 \\
&\lec 2^{\on } \left( 2^{-3\un } \sum_{k=N}^{\un }\sum_{\tilde{Q} \in S_k} \int_0^\infty  \ee^{-|\la |^{\fract{1}{2}} z_3 }  \int_{\RR^2} \chi_{\tilde Q}(z) | u_0 (z) |   \d z'\, \d z_3\right. \\
&\left.\hspace{1cm}+ \sum_{k> \un}2^{-3k } \sum_{\tilde{Q} \in S_k} \int_0^\infty \ee^{-|\la |^{\fract{1}{2}} z_3 }  \int_{\RR^2  }\chi_{\tilde Q}(z) | u_0 (z) |  \d z'\, \d z_3\right),
  \end{split}
   \label{EQ46}  
  \end{equation}
 where we write $z=(z',z_3)$ to emphasize the horizontal and vertical components of $z$ and $\xi\in [x,x_{\Om}]$,
with $[x,x_{\Om}]$ denoting the line segment between the points $x$ and $x_{\Om}$. We also used \eqref{EQ15} in the first inequality above and \eqref{EQ14} in the second. We now apply the Cauchy-Schwarz inequality to the $z_3$-integral to obtain
  \[
   \begin{split}
   &\int_0^1 \ee^{-|\la |^{\fract{1}{2}} z_3 }  \int_{\RR^2 } \chi_{\tilde Q}(z) | u_0 (z) |   \d z'\, \d z_3 \leq \| \ee^{-|\la |^{\fract{1}{2}} z_3 }\|_{L_{z_3}^{2} (0,\infty )} \left\|  \int_{\RR^2  } \chi_{\tilde{Q}} (z',z_3) |u_0 (z',z_3 )|  \d z' \right\|_{L_{z_3}^2 (0,\infty)} 
  \\&\indeq
    \lec | \la |^{-\fract{1}{4}}  \left\| u_0 \right\|_{L^2 (\tilde Q)} 2^{k}
  \lec | \la |^{-\fract{1}{4}}   \left\| u_0 \right\|_{M^{2,q}} 2^{\frac{2+q}{2}k}
  \end{split}
  \]
for every $\tilde{Q} \in S_k$, where we applied the Cauchy-Schwarz inequality in the $z'$-integral in the second inequality.  Substituting this into the above estimate gives 
 \[\begin{split}
 \left| \int_{\HH}  q_{\la ,x,x_{\Om} } (z ) (1-\chi ) u_0 (z)  \d z \right| & \lec  2^{\on}| \la |^{-\fract{1}{4}}   \left\| u_0 \right\|_{M^{2,q}}  \left( 2^{-3\un }\sum_{k=N}^{\un }2^{\frac{2+q}{2}k}+  \sum_{k> \un} 2^{\frac{q-4}{2}k}\right) \\
 &\lec 2^{\on+\frac{q-4}{2}\un } | \la |^{-\fract{1}{4}}   \left\| u_0 \right\|_{M^{2,q}},\end{split}
 \]
 from which the lemma follows by integrating in $\lambda$ and noting that $|\la |\geq \kappa$.
\end{proof}

\cole
\begin{Lemma}[Estimate for $p_{\loc,\Helm}$]
\label{L03}
For every $t>0$
  \begin{equation}
  \| p_{\loc,\Helm} (t) \|_{L^{\fract{3}{2}}(\HH)} \leq C \| u (t) \|^{2}_{L^3 (Q^{***})} 
  .
    \llabel{EQ153}  
  \end{equation}
\end{Lemma}
\colb

\begin{proof}[Proof of Lemma~\ref{L03}]
This follows directly from the Calder\'on-Zygmund estimate applied for each of the two components of the Neumann kernel~\eqref{EQ27}.
\end{proof}

\cole
\begin{Lemma}[Estimate for $p_{\loc,\harm}$]\label{L04}
There exists a function $\theta$ depending only on $t$ such that
  \begin{equation}
   \begin{split}
     & \| p_{\loc,\harm} -\theta\|_{L^{\fract{3}{2}} ((0,t); L^{\fract{17}{10}} (\HH))} 
      \\&\indeq
      \lec  2^{q\on } \left(
       \| \alpha \|_{L^{\fract{39}{5}}(0,t)}^{\fract{13}{34}} \beta(t)^{\fract{21}{34}}+ 2^{-\fract{4}{17}\un}\| \alpha \|_{L^3 (0,t)}^{\fract12}   \beta(t)^{\fract12}+2^{-\fract{21}{17}\un} \| \alpha \|_{L^{\fract32}(0,t)}   \right)
  \end{split}
    \llabel{EQ154}  
  \end{equation}
for all $q>0$, where $\alpha$ and $\beta$ are defined in \eqref{EQ08}.
\end{Lemma}
\colb

Here we briefly comment why we estimate $p_{\loc,\harm}$ in $L^{\fractext{3}{2}}_t L_x^{17/10}$. 
We are interested in estimating a term of the form $\int_Q u p_{\loc,\harm}$ (see Lemma~\ref{L09} below) for a given cube $Q\in \mathcal{C}_N$, for which we can use a bound of the form $\| u \|_{L^{\fractext{3}{2}}_t L^{r'}} \| p_{\loc,\harm}-\theta \|_{L^{3/2}_t L^{r}}$, where $r'$ is the conjugate exponent to $r$. The borderline value of $r$ is $9/5$ as then one can obtain $ \| p_{\loc,\harm } - \theta \|_{L^{\fractext32}_t L^{\fractext{9}{5}}_x} \lec \| \na p_{\loc,\harm }  \|_{L^{\fractext32}_t L^{\fractext{9}{8}}_x} \lec |Q|^{\fractext{q}{3}} \alpha(t)^{2/3} \beta (t)^{2/3} $ (by considering the leading order term only), but since this gives the $L^\infty$ norm of $\alpha$ such an estimate makes it impossible to use an ODE-type argument in the a priori bound. Taking $r<9/5$ replaces the $L^\infty$ norm with a high $L^p$ norm, which makes it possible to use an ODE-type argument, but $r$ also cannot be too low. For example taking $r=8/5$ one can similarly obtain, up to leading order, $ \| p_{\loc,\harm } - \theta \|_{L^{\fractext32}_t L^{\fractext{8}{5}}_x} \lec \| \na p_{\loc,\harm }  \|_{L^{\fractext32}_t L^{\fractext{24}{23}}_x} \lec |Q|^{\fract{q}{3}} \| \alpha \|_{L^{21/5} (0,t)}^{7/16} \beta (t)^{9/16} $, while a Gagliardo-Nirenberg-Sobolev argument for $u$ gives $\| u \|_{L^{3}_t L^{\fractext{8}{3}}_x} \lec |Q|^{q/6}  \| \alpha \|_{L^{15/7} (0,t)}^{15/16} \beta (t)^{9/16}$. In this case the total power of $\beta$ is $9/8>1$, which makes it impossible to absorb it by the dissipation term on the left-hand side of the local energy inequality. It turns out that the choice $r=17/10$ is optimal in this context, as it settles both issues.

\begin{proof}[Proof of Lemma~\ref{L04}]
By the Poincar\'e-Sobolev-Wirtinger inequality 
(cf.~Theorem~II.6.1 in \cite{G})
we have, with $\theta$ depending only on $t$,
  \begin{equation}\| p_{\loc,\harm }(t) - \theta(t) \|_{L^{\fract{17}{10}}(\HH)} \lec \| \na p_{\loc,\harm} (t) \|_{L^{\fract{51}{47}} (\HH)} ,
    \llabel{EQ152}  
  \end{equation}
and thus, using maximal regularity result for half-space (\cite{SvW,GS}),
  \begin{equation}   \llabel{EQ47}  
  \begin{split}
  &\| p_{\loc,\harm } - \theta \|^{\fract{3}{2}}_{L^{\fract32} ((0,t);L^{\fract{17}{10}}(\HH))} 
      \lec \| \mathbb{P} \na \cdot (\chi_{**} u\otimes u ) \|^{\fract32}_{L^{\fract32} ((0,t);L^{\fract{51}{47}}(\HH))} 
  \\&\indeq
     \lec  \int_0^t  \| u (s) \|_{L^{\fract{102}{43}}(Q^{***})}^{\fract32} \|\na (\chi_{**} u )(s) \|_{L^2}^{\fract32}   \d s
    \\&\indeq
      \lec  \int_0^t  \left( \| u (s) \|_{L^{2}(Q^{***})}^{\fract{39}{34}}  \|\na (\chi_{**} u )(s) \|_{L^2}^{\fract{63}{34}} + 2^{-\fract{6}{17}\un } \|  u (s) \|_{L^2(Q^{***})}^{\fract32}  \|\na (\chi_{**} u )(s) \|_{L^2}^{\fract32} \right) \d s
   \\&\indeq
     \lec  \int_0^t  \biggl( \| u (s) \|_{L^{2}(Q^{***})}^{\fract{39}{34}}  \|\na  u (s) \|_{L^2(Q^{***})}^{\fract{63}{34}}+2^{-\fract{6}{17}\un } \| u (s) \|_{L^2(Q^{***})}^{\fract{3}{2}}  \|\na  u (s) \|_{L^2(Q^{***})}^{\fract{3}{2}} 
   \\&\indeq\indeq\indeq\indeq\indeq\indeq\indeq\indeq\indeq\indeq\indeq\indeq\indeq\indeq\indeq\indeq\indeq\indeq\indeq\indeq\indeq\indeq
  + 2^{-\fract{63}{34}\un } \|  u (s) \|_{L^2(Q^{***})}^{3}  \biggr) \d s\\
  &\indeq\lec 2^{\fract{3q}{2}\on} 
  \biggl(
  \left( \int_0^t \alpha(s)^{\fract{39}{5}} \d s \right)^{\fract{5}{68}}\beta(t)^{\fract{63}{68}}
    + 2^{-\fract{6}{17}\un }\left( \int_0^t \alpha(s)^{3} \d s \right)^{\fract14}  \beta(t)^{\fract34}
   \\&\indeq\indeq\indeq\indeq\indeq\indeq\indeq\indeq\indeq\indeq\indeq\indeq\indeq\indeq\indeq\indeq\indeq\indeq\indeq\indeq\indeq\indeq
   +2^{-\fract{63}{34}\un } \int_0^t \alpha(s)^{\fract32} \d s  \biggr)
  ,
  \end{split} 
  \end{equation}
and the proof is concluded, where we used the fact that $2^{3\un } \lec |Q^{***}| \lec 2^{3\on}$.
\end{proof}

\cole
\begin{Lemma}[Estimate for $p_{\nonloc,\Helm}$]\label{L05}
We have
  \begin{equation}
  \| p_{\nonloc,\Helm} (t) \|_{L^{\infty}(\Om )} 
    \lec  2^{\on + (q-4)\un } \|u(t) \|_{M^{2,q}}^2  
  ,
    \llabel{EQ48}  
  \end{equation}
for every $t>0$ and $q\in (0,4)$.
\end{Lemma}
\colb

\begin{proof}[Proof of Lemma~\ref{L05}]
We omit the $t$ variable in the notation. Recall that $Q^*$ is the union of the neighbors of $Q$, which is a cover $\Om$ using cubes from $\mathcal{C}_N$. For $x\in Q^*$
and $z\not \in Q^{**}$ we have
  \begin{equation}
  |\partial_{z_i}\partial_{z_j} N_{x,x_{\Om }} (z) | \lec \frac{|x-x_{\Om }|}{|x-z|^4}\lec\frac{2^{\on } }{|x-z|^4}\lec 2^{\on }  \begin{cases} 2^{-4\un }  &\quad z\in \tilde Q \in S_k, k\leq \un ,\\
   2^{-4k } &\quad z\in \tilde Q \in S_k, k\geq \un+1 ,
   \end{cases}
    \llabel{EQ49}   
  \end{equation}
as in \eqref{EQ14}, \eqref{EQ15}. Thus, for such $x$, \eqref{EQ34} gives
  \begin{equation}\begin{split}
  |p_{\nonloc,\Helm}(x) | &\lec 2^{\on } \sum_{k\geq N} \sum_{\tilde Q\in S_k, \tilde Q\not \subset Q^{**}} \int_{\tilde Q} \frac{|u(z)|^2  }{|x-z|^4} \d z 
 \\&
  \lec 2^{\on } \left( 2^{-4\un } \sum_{k= N}^{\un} \sum_{\tilde Q\in S_k}  \int_{\tilde Q} |u|^2 + \sum_{k\geq \un+1} \sum_{\tilde Q\in S_k} 2^{-4k} \int_{\tilde Q} |u|^2 \right)
   \\&
  \lec\|u \|_{M^{2,q}}^2  2^{\on} \left( 2^{-4\un } \sum_{k= N}^{\un }\sum_{\tilde Q\in S_k}  2^{qk}  + \sum_{k\geq \un +1} \sum_{\tilde Q\in S_k} 2^{(q-4)k}  \right)\lec \|u \|_{M^{2,q}}^2 2^{\on +(q-4)\un} 
   \end{split}
   \llabel{EQ50}  
  \end{equation}
for any $q<4$.
\end{proof}

\cole
\begin{Lemma}[Estimate for $p_{\harm,\geq 1}$]\label{L06}
Let $T>0$. For every $r\in [1,\infty )$, $q\in (0,3)$, $\gamma \in (0,1)$ and $\delta \in (0, \min\{1/r,3-3q/4,3q/2\})$,
  \begin{equation}\begin{split}
  \| p_{\harm, \geq 1} \|_{L^r ((0,T);L^\infty (\Om ))} &\lec_{r,q,\delta,\gamma  }  2^{\on +(q-4)\un  } (1+T)^{\gamma} 
  \\&\indeqtimes
   \left( 2^{\frac{4\delta}3\un }\| \alpha \|_{L^{\frac{r(1-\delta )}{1-r\delta }}(0,T)}^{1-\delta }  \beta(T)^{\delta } +(1+T^{\frac12} 2^{\on q -(1+q)\un} ) \| \alpha \|_{L^r (0,T)}  \right) 
   ,
   \end{split}
    \llabel{EQ51}  
  \end{equation}
where $\alpha$ and $\beta$ are defined in \eqref{EQ08}.
\end{Lemma}
\colb 

\begin{proof}[Proof of Lemma~\ref{L06}]
 Recall that by \eqref{EQ35} we have
  \begin{equation}\begin{split}
   p_{\harm,\geq 1}(x,t) 
      &= \frac{1}{2\pi i} \int_0^t \int_\Gamma \ee^{\la (t-s)} \int_{\HH} q_{\la, x,x_{\Om}} (z)  (1-\chi_*)( F_A(z,s)+F_B(z,s) ) \d z \, \d \la \, \d s \\
   &= p_{\AAA}(x,t)+p_{\BB}(x,t).
  \end{split}
    \llabel{EQ52}
  \end{equation}
In Step~1 below we provide an estimate for $p_A$. Next, in Step~2 we show that $\| f_B \|_{L^\infty (\Om)} \lec |Q|^{(q-2)/3} \| u \|_{M^{2,q}}^2 $ for every $Q\in \mathcal{C}_N$, where $F_B$ is a sum of terms of the form $\nabla'\otimes \nabla' f_B$ where
\eqnb
 m(D') \na' \otimes \na' \int_0^\infty  \left( \left( P(\cdot,|z_3 -y_3|) + P(\cdot,z_3+y_3) \right)\ast ((1-\chi_{**}) v\otimes w ) (y_3) \right)(z',s) \d y_3  =: \nabla'\otimes \nabla' f_B
   \llabel{EQ157}
\eqne
where $v,w$ denote $2D$ vectors whose components are chosen among $u_1,u_2,u_3$ (recall \eqref{EQ30}). We then use this estimate in Step~3 to prove the required estimate on $p_\BB$. \\

\noindent \emph{Step~1.} We show the required estimate for $p_{\AAA}$.\\

For $x\in \Om$, we have, using \eqref{EQ23} with $m=2$, as well as \eqref{EQ14} and \eqref{EQ15},
  \begin{equation}
   \begin{split}
   &\left| \int_{\HH}  q_{\la ,x,x_{\Om } } (z ) (1-\chi_* )F_A (z)  \d z \right| 
   \\&\indeq
   \lec 2^{\on }  \int_{\HH} \frac{\ee^{-|\la |^{\fract{1}{2}} z_3} }{(|\xi'-z'| + \xi_3 +z_3 )^{4}}  |(1-\chi_{*} ) u\otimes u (z)|  \d z  
    +2^{-\un }  \int_{\supp \nabla\chi_*} |q_{\la ,x,x_{\Om }} (z)|  |u (z)|^2  \d z 
        \\&\indeq
          = 2^{\on } \int_1^{\infty} \int_{\RR^2} 
                 \frac{\ee^{-|\la |^{\fract{1}{2}} z_3} }{(|\xi'-z'| + \xi_3 +z_3 )^{4}}  |(1-\chi_{*} ) u\otimes u (z)|  
                 \d z' \,\d z_3 
    \\&\indeq\indeq
         +   2^{\on }\int_0^1 \int_{\RR^2}
              \frac{\ee^{-|\la |^{\fract{1}{2}} z_3} }{(|\xi'-z'| + \xi_3 +z_3 )^{4}}  |(1-\chi_{*} ) u\otimes u (z)|  
         \d z' \,\d z_3 
       \\&\indeq\indeq
      + 2^{-\un } \int_{\supp \nabla\chi_*} |q_{\la ,x,x_Q} (z)|  |u (z)|^2  \d z
     \\&\indeq
      = 2^{\on } (I_1+I_2 +I_3),
   \end{split}
   \llabel{EQ53}
   \end{equation}
   where $\xi$ is a point on the line segment joining $x$ and $x_{\Om }$. We denote the corresponding (pointwise) bound on $|p_A|$ by $p_{\AAA1}+p_{\AAA2}+p_{\AAA3}$,
i.e.,
  \begin{align}
   \begin{split}
    p_{\AAA j}(x,t)
     &=
    \frac{2^{\on }}{2\pi } \int_0^t \int_\Gamma \ee^{-(t-s)\re\lambda} I_j  \d|\lambda| \d s
    ,
   \end{split}
   \label{EQ54}
  \end{align}
for $j=1,2,3$.

For $p_{\AAA1}$ we observe that $\ee^{-|\la |^{\fract{1}{2}} z_3} \leq \ee^{-|\la |^{\fract{1}{2}}}\lec_\gamma |\lambda|^{-\gamma }$ and use \eqref{EQ14} to obtain 
\[
\begin{split}
I_1 &\lec |\la|^{-\gamma } \left( 2^{-4\un}\sum_{k= N}^{\un} \sum_{\tilde Q \in S_k }   \int_{\tilde Q} |u|^2 +\sum_{k\geq \un +1} \sum_{\tilde Q \in S_k } 2^{-4k}  \int_{\tilde Q} | u |^2 \right) \\
&\lec  |\la|^{-\gamma } \| u \|_{M^{2,q}}^2 \left( 2^{-4\un}\sum_{k= N}^{\un} 2^{qk} +\sum_{k\geq \un +1}  2^{(q-4)k}  \right) \\
&\lec  2^{(q-4)\un }|\la|^{-\gamma } \| u \|_{M^{2,q}}^2 .
\end{split}
\]
Thus \eqref{EQ54} gives 
  \begin{equation}\begin{split}
  \| p_{\AAA1} (t) \|_{L^\infty (\Om)} &\lec_\gamma  2^{\on + (q-4)\un} \sup_{s\in [0,t]} \| u(s) \|_{M^{2,q}}^2 \int_0^t (t-s)^{\gamma -1 } \d s\\
  &\lec 2^{\on + (q-4)\un}  \alpha(t) t^\gamma , 
  \end{split}
    \llabel{EQ55}  
  \end{equation}
  for every $t\in [0,T]$, as required.

Next, we bound $p_{\AAA2}$.
Applying H\"older's inequality, we obtain
  \begin{equation}
   \int_\Ga \ee^{(t-s)\re\la } {\ee^{-|\la |^{\fract{1}{2}} z_3} } \d \la \lec_{p,\Gamma } (t-s)^{-\fract1{a}} z_3^{-\fract2{a'}}
   ,
    \llabel{EQ56}  
  \end{equation}
for $1/a+1/a'=1$.
Therefore,  using Tonelli's theorem
  \eqnb\label{EQ57}
  \begin{split}
  |p_{\AAA2}(x,t)| &\leq 2^{\on }  \int_0^t \int_\Gamma \int_0^1 \int_{\RR^2} \ee^{(t-s)\re\la }  \frac{\ee^{-|\la |^{\fract{1}{2}} z_3} }{(|\xi'-z'| + \xi_3 +z_3 )^{4}}  (1-\chi_* )| u\otimes u (z)|   \d z'\, \d z_3 \, \d \la \,\d s  \\
  &\lec_{p}  2^{\on }  \int_0^t (t-s)^{-\fract{1}{a}}  \int_{\RR^2} \int_0^1  \frac{ z_3^{-\fract{2}{a'}}}{(|\xi'-z'| + \xi_3 +z_3 )^{4}}  (1-\chi_* ) |u\otimes u (z)|   \d z_3\, \d z'  \,\d s  
  .
  \end{split}
  \eqne
We write $\RR^2 = \bigcup_{k\geq N} \bigcup_{\tilde Q'\in S_k'} \tilde Q'$, where $\tilde Q'$ denotes the projection of $\tilde Q$ onto $\RR^2$ and $S_k'$ denotes the collection of projections onto $\p\HH$ of the cubes from $S_k$ that touch $\p \HH$. We also set
  \begin{equation}
  p_{\AAA2,\tilde Q}(x,t) = \int_0^t (t-s)^{-\fract{1}{a}}  \int_{\tilde Q' } \int_0^1  \frac{ z_3^{-\fract{2}{a'}}}{(|\xi'-z'| + \xi_3 +z_3 )^{4}}  (1-\chi_* ) |u\otimes u (z)|   \d z_3\, \d z'  \,\d s  ,
    \llabel{EQ58}  
  \end{equation}
so that
  \eqnb\label{EQ59}
  |p_{\AAA2}(x,t)| \leq 2^{\on} \sum_{k\geq N}\sum_{\tilde Q'\in S_k'} p_{\AAA2,\tilde Q} (x,t)
  .
  \eqne
Let
  \begin{equation}
  b= \frac{3}{2\delta }.
    \llabel{EQ60}  
  \end{equation}
The assumption on $\delta $ in the statement and the condition $q<3$ give 
  \eqnb\label{EQ61}
   b> \frac{3r}{2},\qquad b (4-q)>2\qquad \text{ and }\qquad q>\frac1{b}
   .
  \eqne
Furthermore, for each $\tilde Q$, assuming $a'>\max\{2b,1/\gamma \}$ and letting $b'=b/(b-1)<3$, we have
  \begin{equation} \label{EQ62} \begin{split}
   &\int_0^1 \int_{\tilde Q' } | u \otimes u (z) | z_3^{-\fract{2}{a'}} \d z'\, \d z_3 
         \leq 2^{\fract{2}{b}k } 
            \left( \int_0^1 z_3^{-\fract{2b}{a'}} \d z_3\right)^{\fract{1}{b}} \| u \|_{L^{2b'}(\tilde Q)}^{2}
   \\&\indeq
      \lec_{a,b} 2^{\fract{2}{b}k} \left( \| u \|_{L^2 (\tilde Q)}^{\fract{2b-3}{b}} \, \| \na u \|_{L^{2}(\tilde Q)}^{\fract{3}{b}} + 2^{-\fract{3}{b}k } \| u \|_{L^2(\tilde Q)}^2\right)
    ,
    \end{split}
  \end{equation}
where we used the Gagliardo-Nirenberg-Sobolev inequality.
Thus, for $\tilde Q\in S_k$ with $k\geq \un +1$ we obtain, for any $T>0$,
  \begin{equation}\begin{split}
  \|  &p_{\AAA2,\tilde Q} \|_{L^r_t L^\infty_x (\Om \times (0,T))} \\
  &\leq  \left\|  \int_0^t (t-s)^{-\fract1a}  \int_{\tilde Q'} \int_0^1  \frac{ z_3^{-\fract{2}{a'}}}{(|\xi'-z'| + \xi_3 +z_3 )^{4}}  (1-\chi_* )| u\otimes u (z)|   \d z_3\, \d z'  \,\d s  \right\|_{L^r_t L^\infty_x (\Om\times (0,T))} \\
  &\lec_{a,b} 2^{\left(-4+ \frac{2}{b}\right)k}    \left\|  \int_0^t (t-s)^{-\fract{1}{a}}  \left( \| u(s) \|_{L^2 (\tilde Q)}^{\fract{2b-3}{b}} \, \| \na u(s) \|_{L^{2}(\tilde Q)}^{\fract{3}{b}} + 2^{-\fract{3}{b}k } \| u(s) \|_{L^2(\tilde Q)}^2\right) \d s  \right\|_{L^r_t  (0,T)}\\
  &\lec_a  2^{\left(-4+ \frac{2}{b}\right)k}  T^{\fract{1}{a'}} \left( \int_0^T \left( \| u(s) \|_{L^2 (\tilde Q)}^{\fract{2b-3}{b}} \, \| \na u(s) \|_{L^{2}(\tilde Q)}^{\fract{3}{b}} + 2^{-\fract{3}{b}k } \| u(s) \|_{L^2(\tilde Q)}^{2}\right)^{r} \d s  \right)^{\fract{1}{r}}\\
  &\lec_a  2^{\left(-4+ \frac{2}{b}\right)k}  T^{\fract{1}{a'}} \Biggl( \left( \int_0^T \| u(s) \|_{L^2 (\tilde Q)}^{\fract{2r(2b-3)}{2b-3r}} \d s \right)^{\fract{2b-3r}{2br}} \left( \int_0^T \| \na u (s) \|^{2}_{L^2 (\tilde Q) } \d s\right)^{\fract{3}{2b}} 
  \\&\indeq\indeq\indeq\indeq\indeq\indeq\indeq\indeq\indeq\indeq\indeq\indeq\indeq\indeq
  +2^{-\fract{3}{b}k } \left( \int_0^T \| u(s) \|_{L^2 (\tilde Q)}^{2r} \d s\right)^{\fract{1}{r}} \Biggr)\\
  &\lec_a 2^{\left(-4+ \frac{2}{b}+q\right)k}   T^{\fract{1}{a'}} \left( \| \alpha \|_{L^{\fract{r(2b-3)}{2b-3r}}(0,T)}^{\fract{2b-3}{2b}}  \beta(T)^{\fract{3}{2b}} +2^{-\fract{3}{b}k } \| \alpha \|_{L^r (0,T)}  \right)
  ,
  \end{split}
    \llabel{EQ63}  
  \end{equation}
where we used Young's inequality $\| f\ast g \|_r \leq \| f \|_1 \|g\|_r$ 
in $t$ 
in the third inequality (which gives the constraint $a>1$),
and H\"older's inequality in $t$ in the fourth (note that $3r<2b$ by \eqref{EQ61}).
For $k\leq \un $ we obtain a similar estimate, except that $2^{\left(-4+ \frac{2}{b}+q\right)k} $ is replaced by $2^{-4\un +\left(\frac{2}{b}+q\right)k} $. Thus \eqref{EQ59} gives
  \begin{equation}\begin{split}
  \| p_{\AAA2}   \|_{L^r_t L^\infty_x (\Om\times (0,T) )} &\lec_{a,b} 2^{\on } T^{\frac{1}{a'}} \left( \left( 2^{-4\un }\sum_{k=N}^{\un }  2^{\left(\frac{2}{b}+q\right)k} + \sum_{k\geq \un +1} 2^{\left(-4+ \frac{2}{b}+q\right)k}  \right)\| \alpha \|_{L^{\frac{r(2b-3)}{2b-3r}}(0,T)}^{\frac{2b-3}{2b}}  \beta(T)^{\frac{3}{2b}} \right. \\
  &\hspace{3cm}\left.+ \left( 2^{-4\un }\sum_{k=N}^{\un }  2^{\left(q-\frac{1}{b}\right)k} + \sum_{k\geq \un +1} 2^{\left(-4- \frac{1}{b}+q\right)k}  \right) \| \alpha \|_{L^r (0,T)} \right)\\
  &\lec 2^{\on + \left(\frac{2}{b}+q-4\right)\un} T^{\frac{1}{a'}} \left( \| \alpha \|_{L^{\frac{r(2b-3)}{2b-3r}}(0,T)}^{\frac{2b-3}{2b}}  \beta(T)^{\frac{3}{2b}} + 2^{-\frac{3}{b}\un }\| \alpha \|_{L^r (0,T)}  \right)\\
  &= 2^{\on + \left(\frac{4\delta }{3}+q-4\right)\un} T^{\fract{1}{a'}} \left( \| \alpha \|_{L^{\frac{r(1-\delta )}{1-r\delta }}(0,T)}^{1-\delta }  \beta(T)^{\delta } +2^{-{2}{\delta }\un } \| \alpha \|_{L^r (0,T)}  \right),
  \end{split}
    \llabel{EQ64}  
  \end{equation}
as required. Note that 
the infinite sum converges since $-4+2/3+q  <0$ and the finite sum in the second line converges since $q>1/b$, recall~\eqref{EQ61}.

Finally, we bound $p_{\AAA3}$.
First, by \eqref{EQ23}, we have
  \begin{equation}\begin{split}
  I_3 &= 2^{-\on-\un} \int_{\supp \nabla\chi_*} |q_{\la,x,x_{\Om }} (z) | |u(z)|^2\d z 
  \lec 2^{-4\un} \int_{Q^{**}} \ee^{-|\la |^{\fract{1}{2}}z_3} |u(z)|^2\d z 
  ,
  \end{split}
    \llabel{EQ65}  
  \end{equation}
where 
in the second inequality, we used that
$||x'-z'| + x_3 +z_3 |\gec 2^{\un }$ and $|x-x_{\Om}|\lec 2^{\on}$ on $\supp \nabla \chi_*$ (cf.~\eqref{EQ14} and \eqref{EQ15}).
Now, we 
apply the same analysis as for  $p_{\AAA2}$ 
yielding the same bound on $p_{\AAA3}$ as we obtained for $p_{\AAA2}$. 
The only difference is that here we do not need to sum in $\tilde Q\in \mathcal{C}_N$.\\

\noindent \emph{Step~2.} We show that, at each time, $\| f_B \|_{L^\infty (Q)} \lec |Q|^{\fract{q-2}{3}} \| u \|_{M^{2,q}}^2 $ 
for every $Q\in \mathcal{C}_N$. (Analogously we can obtain $\| F_B \|_{L^\infty (Q)} \lec |Q|^{\fract{q-4}{3}} \| u \|_{M^{2,q}}^2 $.)\\

Note that in this step the sets $Q$, $Q^*$, $Q^{**}$, $Q^{***}$ are not related to $\Omega$, but to a fixed cube $Q$. We shall use the estimate
  \begin{equation}\label{EQ66}  
  |m(D') P(y',y_3) | \lec \frac{y_3}{(|y'|+y_3)^{3+\alpha}},
      \end{equation}
where $m(D')$ is a multiplier (in the $y'$ variables)   that is homogeneous of degree $\alpha >-2$, cf.~\cite[p.~576]{MMP1}. 
Let $z\in Q$, and suppose that $Q\in S_n$.  We only consider $P(z_3+y_3)$, as the part with $P(|z_3-y_3|)$ is similar.
We have
  \begin{equation}\begin{split}
  |f_B (z)| &\lec  \sum_{\tilde Q\in \mathcal{C}_N , \tilde Q\not \in Q^{**}} \int_0^\infty \int_{\RR^2 } \frac{z_3+y_3}{(|z'-y'|+z_3+y_3)^3} (1-\chi_{**} ) \chi_{\tilde Q} |u\otimes u (y',y_3 )| \d y' \, \d y_3\\
  &\lec 2^{n} \left(2^{-3n} \sum_{k=N}^{n} \sum_{\tilde Q \in S_k} \| u \|_{L^2 (\tilde Q)}^2 +  \sum_{k\geq n+1} \sum_{\tilde Q \in S_k} 2^{-3k}\| u \|_{L^2 (\tilde Q)}^2 \right)\\
  &\lec 2^{n} \| u \|_{M^{2,q} }^2  \left(2^{-3n} \sum_{k=N}^{n} \sum_{\tilde Q \in S_k} 2^{qn} +  \sum_{k\geq n+1} \sum_{\tilde Q \in S_k} 2^{(q-3)k} \right)\\
  &\lec 2^{(q-2)n} \| u \|_{M^{2,q} }^2 ,
   \end{split}
   \llabel{EQ67}  
  \end{equation}
(recall \eqref{EQ30} for the definition of $F_B=\nabla'\otimes \nabla' f_B$) where, in the second inequality, we used
\[
\frac{z_3+y_3}{(|z'-y'|+z_3+y_3)^3} \lec 2^n \begin{cases} 2^{-3n} &\quad k\leq n,\\
2^{-3k} &\quad k\geq n+1
\end{cases}
\]
whenever $z\in Q$ and $y\in \tilde Q$ for some cube $\tilde Q \in S_k$ that is disjoint with $Q^{**}$ (which is an analogous claim to \eqref{EQ14} and \eqref{EQ15}).\\

\noindent \emph{Step~3.} We show that $\| p_\BB (t) \|_{L^\infty (\Om )} \lec \alpha (t)     2^{\on (1+q) -5 \un}  t^{\fract{1}{2}} $ 
for $q\in (0,3)$.\\

Note that this, together with Step~1, finishes the proof.
We have 
  \begin{equation}\label{EQ68}  
  \begin{split}
  |p_\BB (x,t)| &\lec \int_0^t \int_\Gamma \ee^{(t-s)\re\la } \int_{\HH} \left( \left|  D^2 q_{\la , x, x_{\Om }} (z ) (1-\chi_*) \right| +\left| \na q_{\la , x, x_{\Om }} (z ) \na \chi_* \right|\right.\\
  &\hspace{5cm}\left.+\left|  q_{\la , x, x_{\Om }} (z ) D^2 \chi_* \right|  \right)  \left| f_B(z,s) \right| \d z \, \d |\la| \d s 
  .
  \end{split}
  \end{equation}
Using \eqref{EQ23}, \eqref{EQ14} and \eqref{EQ15}, we get
  \begin{equation}\begin{split}
  &\int_{\HH} \left| D^2 q_{\la , x, x_\Om } (z ) (1-\chi_*) f_B(z,s) \right| \d z \lec 2^{\on }  \sum_{\tilde Q\in \mathcal{C}_N} \int_{\tilde Q} \frac{\ee^{-|\la |^{\fract{1}{2}} z_3}}{(|\xi'-z'|+\xi_3+z_3)^5} (1-\chi_*) |f_B(z,s)| \d z' \d z_3
  \\&\indeq
   \lec 2^{\on }\left(  2^{-5\un }\sum_{k=N}^{\un } \sum_{\tilde Q\in S_k } \int_{\tilde Q} \ee^{-|\la |^{\fract{1}{2}} z_3} |f_B(z,s)| \d z' \d z_3\right.
    \\&\indeq
  \hspace{4cm}\left. \sum_{k\geq \un+1} \sum_{\tilde Q\in S_k }2^{-5k} \int_{\tilde Q} \ee^{-|\la |^{\fract{1}{2}} z_3} |f_B(z,s)| \d z' \d z_3\right)\\
  &\indeq\lec 2^{\on }\left( \int_0^\infty \ee^{- |\la |^{\fract{1}{2}} z_3} \d z_3\right) \| u \|_{M^{2,q}}^2\left(  2^{-5\un }\sum_{k=N}^{\un } 2^{qk }+ \sum_{k\geq \un+1} 2^{(q-5)k} \right)\\
  &\indeq\lec  2^{\on +(q-5 )\un} |\la |^{-\frac{1}{2}}  \| u \|_{M^{2,q}}^2.
    \llabel{EQ69}  
    \end{split}
  \end{equation}  
Similarly, we have
  \begin{equation}\begin{split}
  &\int_{\HH} \left( \left| \na q_{\la , x, x_\Om } (z ) \na \chi_* \right| +\left| q_{\la , x, x_\Om } (z ) D^2 \chi_* \right|\right)  \left| f_B(z,s) \right| \d z \lec 2^{\on -5\un }  \int_{Q^{**}\setminus Q^*}  \ee^{-|\la |^{\fract{1}{2}} z_3} |f_B(z,s)| \d z' \d z_3\\
  &\indeq\lec 2^{\on -5\un} \sum_{k=\un-2}^{\on+2 } \sum_{\tilde Q\in S_k } \int_{\tilde Q} \ee^{-|\la |^{\fract{1}{2}} z_3} |f_B(z,s)| \d z' \d z_3
  \lec 2^{\on -5\un} |\la |^{-\frac12}  \sum_{k=\un-2}^{\on+2 } 2^{qk} \\
  &\indeq\lec  2^{\on -5\un} |\la |^{-\frac12} \| u \|_{M^{2,q}}^2 \sum_{k=\un-2}^{\on+2 } 2^{qk} 
   \lec  2^{\on(q+1) -5\un} |\la |^{-\frac12} \| u \|_{M^{2,q}}^2 
  .
    \llabel{EQ70}  
    \end{split}
  \end{equation}  
Using these estimates in \eqref{EQ68}, we obtain 
  \begin{equation}\begin{split}
  \|p_\BB(t)\|_{L^\infty (\Om )} &\lec  \alpha (t)   2^{\on (1+q) -5 \un}  \int_0^t \int_\Gamma \ee^{-(t-s)\la } |\la |^{-\fract{1}{2}} \d \la \, \d s    \lec \alpha (t)     2^{\on (1+q) -5 \un}  t^{\fract{1}{2}} 
  ,
  \end{split}
    \llabel{EQ71}  
  \end{equation}
and the proof is complete.
\end{proof}

\cole
\begin{Lemma}[Estimate for $p_{\harm,\leq 1}$]
\label{L07}
For every $t\geq 0$ and $q\in(0,3)$, we have
  \begin{equation}
  \| p_{\harm, \leq 1} (t)\|_{L^\infty (\Om )} 
     \lec   2^{\frac{\on}{4} + (q-4)\un} 
        \alpha(t)  t^{\frac38}
    \llabel{EQ72}  
  \end{equation}
\end{Lemma}
\colb
Recall from \eqref{EQ08} that $\alpha (t) =\sup_{s\in[0,t]}\| u (s) 
\|_{M_{\mathcal{C}_N}^{2,q}}^2$.

\begin{proof}[Proof of Lemma~\ref{L07}]
By \eqref{EQ35}, we have
  \begin{equation}
  p_{\harm, \leq 1} (x,t) = \frac{1}{2\pi i}\int_0^t \int_\Gamma \ee^{-(t-s)\la } \int_{\HH} q_{\la } (x'-z', x_3,z_3 ) \chi_* F_B(z,s)' \d z \, \d \la \d s,
    \llabel{EQ73}  
  \end{equation}
where $F_B$ is as in $p_{\harm,\geq 1}$. As in Lemma~\ref{L01}, we use \eqref{EQ23} to obtain
 \begin{equation}\begin{split}
  | p_{\harm, \leq 1} (x,t) | &\lec \int_0^t \int_\Gamma \ee^{-(t-s)\re \la} \int_0^\infty \ee^{-|\la |^{\frac12}z_3}  \int_{\RR^2} (|x'-z'|+x_3+z_3)^{-2}\left| \chi_* F_B (z,s) \right|  \d z'\, \d z_3 \, \d |\la |\, \d s \\
  &\leq \int_0^t \| F_B (s) \|_{L^\infty (Q^*)} \int_\Gamma \ee^{-(t-s)\re \la} \int_0^\infty \ee^{-|\la |^{\frac12}z_3}  \int_{\{|z'| \leq C 2^{\on } \}} (|z'|+x_3+z_3)^{-2}\d z'\,  \d z_3 \, \d |\la |\, \d s \\
  &\lec  2^{(q-4)\un+\frac{\on}{4}} \alpha(t)  \int_0^t  \int_\Gamma \ee^{-(t-s)\re \la} \int_0^\infty \ee^{-|\la |^{\frac12}z_3} (x_3+z_3)^{-\frac14}   \d z_3 \, \d |\la |\, \d s \\
  &\lec  2^{(q-4)\un+\frac{\on}{4}}\alpha(t)  \int_0^t  \int_\Gamma \ee^{-(t-s)\re \la} |\la |^{-\frac38} \d |\la |\, \d s \\
  &\lec  2^{(q-4)\un+\frac{\on}{4}}\alpha(t)  \int_0^t (t-s)^{-\frac58} \d s \\
  &\lec 2^{(q-4)\un+\frac{\on}{4}}\alpha(t)  t^{\frac38} ,
    \llabel{EQ74}  \end{split}
  \end{equation}
for every $x\in \Om $ and $t>0$, where in the third inequality we used the estimate $\| F_B \|_{L^\infty (\tilde Q)} \lec |\tilde Q|^{\fract{q-4}{3}} \| u \|_{M^{2,q}}^2 \lec 2^{(q-4)\un }  \| u \|_{M^{2,q}}^2  $, for every $\tilde Q\subset Q^*$ (recall Step~2 above), as well as the fact $\int_{|y'|\leq a } (|y'|+b )^{-2} \d y' \lec (a/b)^{\frac14}$, as in the proof of Lemma~\ref{L01}.
\end{proof}

\section{A priori bound}
\label{sec04}
We now establish our main a~priori bound for solutions to \eqref{EQ01}
for initial data in $M_{\mathcal{C}_n}^{2,q}$.
For simplicity of notation, we replace the label~$N$ in the previous section 
with~$n$.
We work under the assumption
  \begin{equation}
    1\leq q \leq 2
   .
    \label{EQ75}  
  \end{equation}
Recall from \eqref{EQ08} the notation
  \begin{equation}
  \alpha_n (t) =\sup_{s\in[0,t]}\| u (s) \|_{M_{\mathcal{C}_n}^{2,q}}^2 
  \quad \text{ and } \quad 
  \beta_n(t) = \sup_{Q\in \mathcal C_n} \frac 1 {|Q|^{\fract{q}{3}}}\int_0^t\int_{Q}|\nabla u(x,s)|^2\,\d x\,\d s
   .
    \llabel{EQ76}  
  \end{equation}
Since in Lemma~\ref{L15} below we show 
that $\alpha$ and $\beta$ are 
continuous functions of~$t$, we 
define $\alpha_n(0)=\| u_0\|_{M^{2,q}_{\mathcal C_n}}^2 $. 

Theorem~\ref{T02} follows from the following statement.

\cole
\begin{Proposition}[A priori bound]\label{Pro01}
Assume that \eqref{EQ75} holds. There exists $\gamma\geq1$ with the following property. Let $n\in{\mathbb N}$, suppose that $u_0\in \mathring M^{2,q}_{\mathcal C_n}$, and assume that $(u,p)$ is a local energy solution on $\HH\times (0,\infty)$ with the initial data $u_0$ such that   
  \EQ{
   \esssup_{0<s<t}(\alpha_n(s)+\beta_n(s))<\infty,
   \llabel{EQ77}
  }
for all $t<\infty$. Let $T=T_n$ be the solution of 
  \begin{equation}
   a=
       (1+T)^{\gamma}
       \left( 
        (2 a)^{3} T^{\fract{3}{8}}
        +  (2 a)^{\fract{3}{2}} T^{\fract{3}{16} }
        +  a 2^{-n} T
       \right)
      ,
    \label{EQ78}  
  \end{equation}
where $a= \| u_0 \|_{M^{2,q}_{\cn}}^2$. Then
  \EQ{
  \alpha_n(T_n)
         + \beta_n(T_n)
      \leq 
       C \alpha_n(0)
  ,
   \llabel{EQ79}
  }
where $C>0$ is a constant.
\end{Proposition}
\colb
\colb

We note that the proof below shows that the constant in \eqref{EQ11} can be arbitrarily close to~$1$.
Note that $\| u_0\|_{M^{2,2}_{\mathcal C_n}}\to0$ as $n\to \infty$ for $u_0\in \MC$.
Therefore, the solution $T=T_n$ of \eqref{EQ78} satisfies $T_n\to \infty$ as $n\to \infty$.

First, we recall the following lemma, which is adapted from \cite{BKT} to the case of~$\HH$.

\cole
\begin{Lemma}
\label{L08}
Let $u\colon\HH\times (0,T) \to \HH$.
Given $\epsilon>0$,  we have 
  \begin{align}
   \begin{split}
    \frac {1} {|Q|^{\fract{1}{3}}} \int_0^t\int_{Q} |u|^3\,\d x\,\d s 
     &\leq 
      C_\epsilon   |Q|^{q-\fract{4}{3}} \int_0^t \bigg( \frac 1 {|Q|^{\fract{q}{3}}} \int_{Q}|u(s)|^2 \,\d x \bigg)^{3} \,\d s  
    \\&\indeq 
     +  
    \epsilon \int_0^t  \int_{Q}|\nabla u|^2 \,\d x\,\d s
   \\ &\indeq 
   + C |Q|^{\fract{q}{2}-\fract{5}{6}} \int_0^t \bigg( \frac 1 {|Q|^{\fract{q}{3}}} \int_Q |u|^2\,\d x \bigg)^{\fract{3}{2}}\,\d s
   \comma t\in(0,T)  
  ,
  \end{split}   
  \label{EQ80}
  \end{align}
for any cube $Q\subset \HH$.
\end{Lemma} 
\colb

This statement is proven in the case of $\mathbb R^3$ using the Gagliardo-Nirenberg-Sobolev inequality in~\cite{BKT}; the proof of the analog for $\HH$ is straight-forward and is thus omitted.
The inequality \eqref{EQ80} implies that
  \eqnb\label{EQ156}
  \frac {1} {|Q|^{\fract{1}{3}}} \int_0^t\int_{Q} |u|^3
     \leq 
      C_\epsilon |Q|^{q-\fract{4}{3}}  \| \alpha \|_{L^3 (0,t)}^3  +\epsilon |Q|^{\fract{q}{3}} \beta(t) +  |Q|^{\fract{q}{2}-\fract{5}{6}} \| \alpha \|_{L^{\frac32} (0,t)}^{\frac32}  
  ,
  \eqne
for every $Q\in \cn$ and $t>0$, 
suppressing the dependence of $\alpha$ and $\beta $ on~$n$ in the rest of the section.
Similarly, one can show that  
  \eqnb\label{EQ82}
   \| u \|_{L^3 ((0,t);L^{\fract{17}{7}}(Q))} 
       \lec  |Q|^{\frac{q}{6}}  \| \alpha \|_{L^{\frac{75}{41}}(0,t)}^{\frac{25}{68}}  \beta(t)^{\frac9{68}} 
               + |Q|^{\frac{q}{6}-\frac{3}{34}} \| \alpha \|_{L^{\frac{3}2}(0,t)}^{\frac12}  
   .
   \eqne
We now use the pressure estimates developed in the previous section to deduce a bound on the pressure term appearing in the local energy inequality.

\cole
\begin{Lemma}\label{L09}
Let $\phi_Q$ be such that $\phi_Q=1$ on $Q$, $\phi_Q=0$ outside $Q^*$,
and $|\na \phi_Q |\lec |Q|^{-\fract{1}{3}}$. 
Then 
  \begin{equation}
  \int_0^t \int p u\cdot \na \phi_Q 
     \leq 
     C_\epsilon |Q|^{\frac{q}3} (1+t)^{\gamma}
          \left( \| \alpha \|_{L^{8}(0,t)}^{3} 
              + \| \alpha \|_{L^{8}(0,t)}^{\frac{3}{2}}  
              + |Q|^{-\frac16} \| \alpha \|_{L^{8}(0,t)} 
     \right)
                 + \epsilon |Q|^{\frac{q}3} \beta (t)
    ,
    \llabel{EQ81}  
  \end{equation}
for $q\leq 2$, $\epsilon >0$, where $\gamma\geq1$.
\end{Lemma}
\colb

\begin{proof}[Proof of Lemma~\ref{L09}]
We have
  \begin{align}
  \begin{split}
    &  \int_0^t \int (p_{\li,\loc} +p_{\li,\nonloc})u\cdot \na \phi_Q 
    \\&\indeq
    \lec
    |Q|^{-1/3}
    \int_{0}^{t}
     \Vert p_{\li,\loc}\Vert_{L^2(Q^{*})}
     \Vert u\Vert_{L^2(Q^{*})}
     \d s
    +
    |Q|^{-1/3}
    \int_{0}^{t}
     \Vert p_{\li,\nonloc}\Vert_{L^\infty(Q^{*})}
     \Vert u\Vert_{L^1(Q^{*})}
     \d s
    \\&\indeq
    \lec
    |Q|^{\frac{q-1}{6}}
    \int_{0}^{t}
     \| u_0 \|_{M_{\cn}^{2,q}} 
     \Vert u\Vert_{L^2(Q^{*})}
     s^{-\frac34}
     \d s
    +
    |Q|^{\frac{q-4}{6}}
    \int_{0}^{t}
     \| u_0 \|_{M_{\cn}^{2,q}} 
     \Vert u\Vert_{L^1(Q^{*})}
     s^{-\frac34}
     \d s
    \\&\indeq
    \lec 
    |Q|^{\frac{q-1}{6}} 
     \int_0^t 
        \| u \|_{L^2(Q^*)} 
        \| u_0 \|_{M_{\cn}^{2,q}} 
     s^{-\frac34} \d s 
    \\&\indeq
    \lec 
    |Q|^{\frac{q}{3}-\frac16} 
    \| u_0 \|_{M_{\cn}^{2,q}} 
    \int_0^t 
    \alpha(s)^{1/2} 
    s^{-\frac34} \d s
    \lec 
\colb
      |Q|^{\frac{q}{3}-\frac16}
      t^{\fract{3}{16}} 
      \| u_0 \|_{M_{\cn}^{2,q}} 
     \| \alpha \|_{L^8(0,t)}^{\frac{1}{2}}
\colb
   ,
   \end{split}
   \label{EQ84}  
  \end{align}
where we used  
the first two inequalities of 
\eqref{EQ37} in the second step.

Next, by the third estimate in~\eqref{EQ37}, we may bound
  \begin{equation}\begin{split}
  \int_0^t \int p_{\loc,\Helm} \,u\cdot \na \phi_Q &\lec |Q|^{-\frac13} \left( \int_0^t \int_{Q^*} |u|^3 + \int_0^t \int_{Q^*} |p_{\loc,\Helm}|^{\frac32} \right)
   \lec |Q|^{-\frac13}  \int_0^t \int_{Q^{***}} |u|^3 
   \\&
   \lec
\colb
   C_\epsilon |Q|^{\frac{q}{3}}|Q|^{\frac{2q}{3}-\fract{4}{3}}  \| \alpha \|_{L^3 (0,t)}^3  
        +\epsilon |Q|^{\fract{q}{3}} \beta(t) 
      +  |Q|^{\frac{q}{3}}|Q|^{\fract{q}{6}-\fract{5}{6}} \| \alpha \|_{L^{\frac32} (0,t)}^{\frac32}  
\colb
   ,
  \end{split}
    \label{EQ85}  
  \end{equation}
by \eqref{EQ156}.
With $\theta=\theta(t)$ as in \eqref{EQ37}, we write
  \begin{equation}\begin{split}
  &\int_0^t \int p_{\loc,\harm} \,u\cdot \na \phi_Q =\int_0^t \int (p_{\loc,\harm} - \theta) \,u\cdot \na \phi_Q   
  \\&\indeq
   \lec |Q|^{-\frac13} \| u \|_{L^3((0,t);L^{\frac{17}7} (Q^*))}  \| p_{\loc,\harm} -\theta\|_{L^{\frac32} ((0,t); L^{\frac{17}{10}} (\HH))}
   \\&\indeq
    \lec 
     |Q|^{\frac{q}{3}}
     \left(  |Q|^{\frac{q-2}{6}} \| \alpha \|_{L^{\frac{75}{41}}(0,t)}^{\frac{25}{68}}  \beta(t)^{\frac9{68}} 
           + |Q|^{{\fract{q}{6}-\fract{43}{102}}} \| \alpha \|_{L^{\frac{3}2}(0,t)}^{\frac12}   \right) 
   \\&\indeqtimes
  \left(
       \| \alpha \|_{L^{\frac{39}{5}}(0,t)}^{\frac{13}{34}} 
       \beta(t)^{\frac{21}{34}}
      + |Q|^{-\frac{4}{51}}\| 
            \alpha \|_{L^3 (0,t)}^{\frac12}   
            \beta(t)^{\frac12}
      +|Q|^{-\frac{21}{51}}
         \| \alpha \|_{L^{\frac32}(0,t)}  
  \right) 
  \\&\indeq
 \lec 
  |Q|^{\frac{q}3} |Q|^{\frac{q-2}{6}}
     \left(  
             t^{\frac{253}{1632}}\| 
             \alpha \|_{L^{8}(0,t)}^{\frac{25}{68}}  
             \beta(t)^{\frac9{68}} 
           + |Q|^{-\fract{3}{34}} 
             t^{\frac{13}{48}}
             \| \alpha \|_{L^{8}(0,t)}^{\frac12}   \right) 
   \\&\indeqtimes
    \left(
     t^{\frac{1}{816}}
      \| \alpha \|_{L^{8}(0,t)}^{\frac{13}{34}} \beta(t)^{\frac{21}{34}}
      + 
       |Q|^{-\frac{4}{51}}
       t^{\frac{5}{48}}
       \| \alpha \|_{L^8 (0,t)}^{\frac12}   
       \beta(t)^{\frac12}
      +
        |Q|^{-\frac{21}{51}} 
        t^{\frac{13}{24}}
        \| \alpha \|_{L^{8}(0,t)}  
      \right) 
   ,
  \end{split}
   \label{EQ86}
  \end{equation}
where we applied \eqref{EQ82} in the second  inequality. 
Therefore,
  \begin{equation}\begin{split}
  &\int_0^t \int p_{\loc,\harm} \,u\cdot \na \phi_Q 
   \\&\indeq
   \lec  
       |Q|^{\frac{q}{3}} 
       (1+t)
     \biggl(   
         \Vert \alpha \Vert_{L^{8}(0,t)}^{\frac{3}{4}} 
         \beta(t)^{\frac{3}{4}} 
         + 
        \Vert \alpha \Vert_{L^{8} (0,t)}^{\frac{59}{68}}   
        \beta(t)^{\frac{43}{68}} 
         +
      \Vert \alpha \Vert_{L^{8}(0,t)}^{\fract{93}{68}}
      \beta(t)^{\frac9{68}} 
       \\&\indeq\indeq\indeq\indeq\indeq\indeq\indeq\indeq\indeq
        + 
       \Vert \alpha \Vert_{L^{8}(0,t)}^{\frac{15}{17}} 
       \beta(t)^{\frac{21}{34}}
       + 
        \Vert \alpha \Vert_{L^{8} (0,t)} 
        \beta(t)^{\frac12}
         +
         \Vert \alpha \Vert_{L^{8}(0,t)}^{\frac32}   
     \biggr) 
  \\&\indeq
\colb
     \lec 
        C_\epsilon 
       |Q|^{\frac{q}3} (1+t)^{C} 
        \left( \Vert \alpha \Vert_{L^{8}(0,t)}^{3} + \Vert \alpha \Vert_{L^{8}(0,t)}^{\frac{3}{2}}  \right)  
  +  
   \epsilon |Q|^{\frac{q}{3}} \beta(t) ,
  \end{split}
\colb
    \llabel{EQ87}  
  \end{equation} 
where we used $|Q|\gec 1$ to remove $|Q|^{\frac{q-2}{6}}$ and other non-positive powers of $|Q|$ in the parentheses in the first inequality.

Next,
  \begin{equation}\begin{split}
    & \int_0^t \int p_{\nonloc,\Helm} \,u\cdot \na \phi_Q 
      \lec |Q|^{-\frac13}  
          \int_0^t
            \Vert u(s) \Vert_{L^1(Q^{*})} \Vert p_{\nonloc,\Helm}(s) 
            \Vert_{L^\infty (Q^{*})} 
          \d s
  \\&\indeq
  \lec  |Q|^{\fract{q}3-\fract{5}{6}}  \int_0^t\Vert u(s) \Vert_{L^2(Q^{*})}  
    \Vert u(s) \Vert_{M_{\cn}^{2,q}}^2 \d s  
  \lec     
   |Q|^{\frac{q}2-\fract{5}{6}}   
   \int_0^t \alpha^{\frac32}(s)  \d s  
  \\&\indeq
   =
\colb
      |Q|^{\frac{q}{3}}
      |Q|^{\frac{q}6-\fract{5}{6}} 
      \Vert \alpha \Vert_{L^{\frac32}(0,t)}^{\frac32}  
\colb
  .
   \end{split}
   \label{EQ88}  
  \end{equation}

Using 
the estimate for 
$p_{\harm,\geq 1}$ in
\eqref{EQ37} with
$\gamma = 1/2$, $r=3/2$ and $\delta=1/2$, we have
  \begin{equation}\begin{split}
    &  \int_0^t \int p_{\harm,\geq 1} \,u\cdot \na \phi_Q 
       \lec |Q|^{-\frac13} \Vert u \Vert_{L^3((0,t);L^1 (Q^*))}  \Vert p_{\harm,\geq 1} \Vert_{L^{\frac32} ((0,t); L^{\infty} (Q^* ))}
    \\&\indeq
       \lec |Q|^{\frac16} \Vert u \Vert_{L^3((0,t);L^2 (Q^*))}  
            \Vert p_{\harm,\geq 1} \Vert_{L^{\frac32} ((0,t); L^{\infty} (Q^* ))}
    \\&\indeq
      \lec  
         |Q|^{\frac{1}{6}+\fract{q}{6}} \Vert \alpha  \Vert^{\fract{1}{2}}_{L^{\fract{3}{2}}(0,t)}   \Vert p_{\harm,\geq 1} \Vert_{L^{\frac32} ((0,t); L^{\infty} (Q^* ))}
   \\&\indeq
     \lec  |Q|^{\fract{q}{2}-\fract{11}{18}} (1+t)^{\frac12} \Vert \alpha  \Vert^{\fract{1}{2}}_{L^{\fract{3}{2}}(0,t)}      
           \left( 
              \Vert \alpha \Vert_{L^{3}(0,t)}^{\fract{1}{2} }  \beta(t)^{\fract{1}{2} } +(1+t^{\frac12})\Vert \alpha \Vert_{L^{3/2} (0,t)}  
         \right)
   \\&\indeq
     \lec 
\colb
       C_\epsilon |Q|^{\frac{q}{3}}|Q|^{\frac{q}{3}-\fract{11}{9}} (1+t)^C  \Vert \alpha \Vert^{\frac32}_{L^{3} (0,t)} 
       + \epsilon |Q|^{\frac{q}{3}} \beta (t)
\colb
  ,
  \end{split}
    \llabel{EQ89}  
  \end{equation}
and, from the last bound in \eqref{EQ37},
  \begin{equation}\begin{split}
  &
  \int_0^t \int p_{\harm,\leq 1} \,u\cdot \na \phi_Q \lec |Q|^{-\frac13}  \int_0^t\| u(s) \|_{L^1(Q^*)} \| p_{\harm,\leq 1}(s) \|_{L^\infty (Q^*)} \d s
  \\&\indeq
    \lec |Q|^{\frac{q}{3}-\fract{13}{12}}  
       \int_0^t\| u(s) \|_{L^2(Q^*)} \alpha(s) s^{\frac{3}{8}}  \d s
  \lec |Q|^{\frac{q}2-\frac{13}{12}}   \int_0^t \alpha(s)^{\frac32}s^{\frac{3}{8}} \d s 
  \\&\indeq
\colb
  \lec  
    |Q|^{\frac{q}{3}}    |Q|^{\frac{q}6-\frac{13}{12}} t^{\frac{19}{16}}  \| \alpha \|_{L^8 (0,t)}^{\frac32}
\colb
  ,
  \end{split}
    \label{EQ90}  
  \end{equation}
and the proof is complete.
\end{proof}

The following lemma contains the necessary barrier statement needed for the global existence.

\cole
\begin{Lemma}\label{L10}
Suppose that $f\in L^\infty_{\loc}  ([0,T_0);[0,\infty ))$ 
for  $\overline{p},\overline{q} >1$, and $p\in [1,\infty)$
satisfy
  \begin{equation}
    f(t) 
     \leq 
       a
      + b(1+t)^\gamma  
      \left( 
        \Vert f \Vert_{L^p (0,t)}^{\overline{p}} 
        +         \Vert f \Vert_{L^p (0,t)}^{\overline{q}} 
        +      c  \Vert f \Vert_{L^p (0,t)}
     \right),
    \llabel{EQ91}   
  \end{equation}
where $\gamma\geq0$ and $a,b,c> 0 $.
Then 
  \begin{equation}
   f(t)\leq 2a   
   ,
   \llabel{EQ92}
  \end{equation}
 for $t\leq \min\{T,T_0\}$, where $T>0$ is the solution of
  \begin{equation}
   a=
       b(1+T)^{\gamma}
       \bigl( 
        (2a)^{\overline{p}} T^{\fract{\overline{p}}{p}}
        +  (2a)^{\overline{q}} T^{\frac{\overline{q}}{p}} 
        + 2 a c T
       \bigr).
    \label{EQ93}  
  \end{equation}
 \end{Lemma} 
\colb

Observe that $T\to \infty$ if $a\to 0$ and $c\to 0$.

\begin{proof}[Proof of Lemma~\ref{L10}]
By \eqref{EQ93}, the function $g(t)=2a$ satisfies
  \begin{equation}
    g(t) 
     \geq 
       a
      + b(1+t)^\gamma  
      \left( 
        \Vert g \Vert_{L^p (0,t)}^{\overline{p}} 
        +         \Vert g \Vert_{L^p (0,t)}^{\overline{q}} 
        +       c \Vert g \Vert_{L^p (0,t)}
     \right),
   \llabel{EQ94}
  \end{equation}
for $t\in[0,T_1]$, where $T_1=\min\{T,T_0\}$. The inequality $f(t)\leq g(t)$ for $t\in [0,T_1]$ follows by a standard barrier argument.
\end{proof}

\begin{proof}[Proof of Proposition~\ref{Pro01}]
Let $Q\in \mathcal{C}_n$. 
Using the local energy inequality \eqref{EQ155} with $\phi(x,t)=\phi_Q(x)\psi_m(t)$ 
where $\psi_m$ is a suitable sequence of functions,
and weak continuity in time, we obtain 
  \begin{equation} 
  \begin{split}
  &\int |u(t)|^2 \phi_Q + 2\int_0^t \int |\na u |^2 \phi_Q \leq  \int |u(0) |^2 \phi_Q + \int_0^t \int |u|^2 \De \phi_Q +\int_0^t (|u|^2 + p ) u\cdot \na \phi_Q 
  \end{split}
   \label{EQ95}
  \end{equation}
from where, using \eqref{EQ156} and Lemma~\ref{L09}. 
  \begin{equation} 
  \begin{split}
  &\int |u(t)|^2 \phi_Q + 2\int_0^t \int |\na u |^2 \phi_Q 
  \\&\indeq
   \leq 
    C  |Q|^{\frac{q}3}\alpha (0)  
      +C |Q|^{\frac{q}{3}-\frac{2}{3}}\int_0^t \alpha (s) \d s
      + C |Q|^{\frac{q}{3}-\frac{1}{6}}       t^{\fract{3}{16}} \alpha(0)^{1/2} \| \alpha \|_{L^{8}(0,t)}^{1/2} 
   \\&\indeq\indeq
    +   C_\epsilon |Q|^{\frac{q}3} (1+t)^{C}
          \left( \| \alpha \|_{L^{8}(0,t)}^{3} 
              + \| \alpha \|_{L^{8}(0,t)}^{\frac{3}{2}}  
          \right)
              + \epsilon |Q|^{\frac{q}3} \beta (t) 
   ,
  \end{split}
    \llabel{EQ96}  
  \end{equation}
for all $t>0$.
Dividing by $|Q|^{\fract{q}{3}}$, taking $\sup_{Q\in \cn }$, and absorbing the last term on the right-hand side, we obtain
  \begin{equation} 
  \begin{split}
    \alpha (t) + \beta (t) 
         &\leq 
            C  \alpha (0) 
           + C t^{\frac{1}{8}}2^{-2n} \Vert \alpha \Vert_{L^{8}(0,t)}
           + C 2^{-\frac{n}{2}}     t^{\fract{3}{16}}   \alpha(0)^{1/2} \| \alpha \|_{L^{8}(0,t)}^{1/2} 
         \\&\indeq
          +C (1+t)^{\gamma} 
        \Bigl( \Vert \alpha \Vert_{L^{8}(0,t)}^{3} 
           + \Vert \alpha \Vert_{L^{8}(0,t)}^{\frac{3}{2}}  
         \Bigr)
   ,
  \end{split}
   \llabel{EQ98}
  \end{equation}
 where $\gamma\geq 0$ is a constant and we used $|Q|\gtrsim
 2^{3n}$. Applying the Cauchy-Schwarz inequality on the third term, we get
  \begin{equation} 
  \begin{split}
    \alpha (t) + \beta (t) 
         &\leq 
            C  \alpha (0) 
          +C (1+t)^{\gamma} 
        \Bigl( \Vert \alpha \Vert_{L^{8}(0,t)}^{3} 
           + \Vert \alpha \Vert_{L^{8}(0,t)}^{\frac{3}{2}}  
           + 2^{-n}\Vert \alpha \Vert_{L^{8}(0,t)}
         \Bigr)
   ,
  \end{split}
   \label{EQ97}
  \end{equation}
The claim now follows from Lemma~\ref{L10}, applied with 
$f(s)= \alpha (s) + \beta(s) $, 
$a= C_0\alpha(0) + C_0 2^{-\frac{3n}{2}}$, 
$b = C_0$,
$c=2^{-n}$
$p=8$,
$\overline{p}=3$,
and $\overline{q}=3/2$,
where $C_0$ is the  constant in~\eqref{EQ97}. Note that the definition \eqref{EQ93} of $T$ given by the lemma then becomes 
\eqref{EQ78}, as required. 
\end{proof}

\begin{proof}[Proof of Theorem~\ref{T01}]
Theorem~\ref{T02} follows by showing that
the solution $T$ of \eqref{EQ78} satisfies 
\eqref{EQ10}.
\end{proof}

\begin{Remark}\label{R01}
{\rm
Using the Gagliardo-Nirenberg inequality, we have
  \begin{equation}
     \frac 1 {|Q|}\int_0^{T_n}\int_Q |u|^3\,dx\,dt 
       \lec
         T_n^{\frac{1}{4}} 
         |Q|^{\fract{q-2}{2}}
         \alpha(T_n)^{\fract{3}{4}}
         \beta(T_n)^{\fract{3}{4}}
       +
|Q|^{\frac{q-3}2}
               \Vert \alpha\Vert_{L^{3/2}(0,T_n)}^{\fract{3}{2}}
      .
    \llabel{EQ99}  
  \end{equation}
Using the bound \eqref{EQ11}, we get
  \begin{align}
   \begin{split}
     \sup_{Q\in \mathcal C_n} 
         \frac 1 {|Q|^{\frac{q}{2}}}\int_0^{T_n}\int_Q |u|^3\,dx\,dt 
      \lec
      T_n^{\fract{1}{4}} 
      \|u_0\|_{\MMnq}^3 + \frac{
                          T_n
                             }{
                          2^{3n/2}
                         }
                         \|u_0\|_{\MMn}^3
   ,
   \end{split}
   \llabel{EQ100}
  \end{align}
where $n$ and $T_n$ are as in Theorem~\ref{T02}.
Similarly, we have
  \begin{equation}
     \frac 1 {|Q|^{\fract{10}{9}}}\int_0^{T_n}\int_Q |u|^{\fract{10}{3}}\,dx\,dt 
       \lec
         |Q|^{\frac{5q-10}{9}}
         \alpha(T_n)^{\fract{2}{3}}
         \beta(T_n)
       +
         |Q|^{\frac{5q-16}{9}}
         \Vert \alpha\Vert_{L^{5/3}(0,T_n)}^{5/3}
      ,
   \llabel{EQ101}
  \end{equation}
which gives
  \begin{align}
   \begin{split}
     \sup_{Q\in \mathcal C_n} \frac 1 {|Q|^{\frac{5q}{9}}}\int_0^{T_n}\int_Q |u|^{\fract{10}{3}}\,dx\,dt 
      \lec
          \left(
          1
         + 
         \frac{T_n}{2^{2n}}
          \right)
        \|u_0\|_{\MMnq}^{\fract{10}{3}}
   \end{split}
   \llabel{EQ102}
  \end{align}
by analogous arguments. 
}
\end{Remark}

\section{Stability}
\label{sec05}

In this section, we prove the stability theorem.
In the following, we use the notation
  \begin{equation}
      \|(f,g)\|_X = 2\max \{\|f\|_X,\|g\|_X\}   
   .
   \llabel{EQ103}
  \end{equation}

\begin{proof}[Proof of~Theorem~\ref{T03}]
Fix $n\in \NN$. Since $u_0^{(k)}\to u_0$ in $\MM$ and by \eqref{EQ06}, we have that for any $\epsilon>0$, there exists $K(n)$ so that 
\[
\| u_0^{(k)} \|_{\MMn} < \epsilon,
\]
for all $k\geq K(n)$. Let  $T_n$ be given by \eqref{EQ78} for
$\|u_0\|_{\MMn}$, and let $T_n^{(k)}$ be the same for
$\|u_0^{(k)}\|_{\MMn}$. The above observation implies that, for
sufficiently large $k$, we have $T_n^{(k)}\geq T_n/2$. For the
finitely many remaining $k$ we may not have $T_n^{(k)} \geq T_n/2$,
but we do have $T_{n'}^{(k)} \geq T_n/2$ by taking $n'$ sufficiently
large due to inclusion in $\MC$. Hence, we can apply Theorem~\ref{T02}
to conclude that $u^{(k)}$ are uniformly bounded in $L^\I(0,T_n/2;
L^2(B_n\cap \HH)) \cap L^2(0,T_n/2;H^1(B_n\cap \HH))$, where
$\{B_n\}_{n\geq 1}$ is a sequence of expanding balls
(e.g.~$B_n=B(0,n)$). Note that $T_n$ is an unbounded, non-decreasing
sequence.

We  need to show that $\partial_t u^{(k)}$ are uniformly bounded in the dual of $L^5(0,T_n/2; W_0^{1,3}(B_n))$.  Let $Q$ be a cube containing $B_n$.  Based on local estimates for $u^{(k)}$ and pressure estimates in Section \ref{sec03}, this follows nearly identically to \cite[p.~561]{MMP1}. The only added work here involves  our treatment of $p^{(k)}_{\loc,\harm}$ for which we estimate
\EQ{
\bigg| \int_0^{T_n/2} \int_{\BB_n} p^{(k)}_{\loc,\harm} \nabla \cdot \phi\,dx\,dt   \bigg| &\leq \bigg(  \int_0^{T_n/2}\bigg(  \int_Q |p^{(k)}_{\loc,\harm}|^{17/10}\,dx\bigg)^{\frac {10} {17} \frac 3 2} \,dt  \bigg)^{2/3} 
\\&\qquad \cdot \bigg( \int_0^{T_n/2}\bigg(  \int_Q |\nabla \phi |^{\frac {17} 7}\,dx\bigg)^{\frac {7} {17} 3} \,dt\bigg)^{\frac 1 3}
\\&\lesssim_{Q,T_n} \bigg(  \int_0^{T_n/2}\bigg(  \int_Q |p^{(k)}_{\loc,\harm}|^{17/10}\,dx\bigg)^{\frac {10} {17} \frac 3 2} \,dt  \bigg)^{2/3} \| \nabla \phi \|_{L^5(0,T_n/2; L^3(B_n))},
   \llabel{EQ05}
}
where we have used H\"older's inequality in both space and time variables---note that $\phi$ is as in \cite{MMP1}. The terms involving $p^{(k)}_{\loc,\harm}$ are uniformly bounded by Lemma~\ref{L04}. Hence,  following \cite[p.~561]{MMP1}, we have obtained the claimed uniform bound on $\partial_t u^{(k)}$.

By repeatedly applying the Lions-Aubin lemma and using a diagonalization argument, we obtain that there exists   $u\colon\HH\times (0,\I) \to \mathbb R^3$  such that, for every $n\in \NN$,
  \begin{equation}\begin{split}
   &u^{(k)} \to  u \quad\text{ weakly-$*$} \text{ in }L^\I(0,T_n/2; L^2(B_n)),
   \\& u^{(k)} \to  u \quad \text{ weakly in }L^2(0,T_n/2;H^1(B_n)),
   \\& u^{(k)}\to u\quad \text{ strongly in }L^2(0,T_n/2;L^2(B_n)),
   \end{split}
   \llabel{EQ104}
   \end{equation}
after passing to a subsequence of $\{u^{(k)}\}$.
By interpolation, the strong convergence can be extended to 
  \begin{equation}\label{EQ105}
   u^{(k)}\to u \text{ strongly in }L^q(0,T_n/2;L^p(B_n)),
  \end{equation}
for any $p,q>2$ such that $2/q+3/q<3/2$ and $p<6$.  By
  Remark~\ref{R01}, the convergence in $L^3(0,T_n/2;L^3(B_n))$ implies  
  \begin{equation}
   \sup_{Q\in \mathcal C_n} \frac 1 {|Q|}\int_0^{T_n/2}\int_Q |u|^3\,dx\,dt \leq  C T_n \|u_0\|_{\MMn}^3 
            + \frac{C T_n}{2^{3n/2}} \|u_0\|_{\MMn}^2,
   \llabel{EQ106}  
  \end{equation}
since this estimate is satisfied by all $u^{(k)}$ for sufficiently large $k$.

 Next, with $\alpha$ and $\beta$ in~\eqref{EQ08} for $u$, we check that 
$\alpha(T)+\beta(T)<\I$ for every $T>0$.   Given $T$, choose $n$ so that $T<T_n/2$ where $T_n$ is as above. The convergence properties listed above on $B_{n'}\times (0,T_{n'}/2)$ for $n'\geq n$ extend to any $Q\in \mathcal C$. Hence, for any such $Q$,
\EQ{
	\sup_{0<s<T} \frac 1 {|Q|^{1/3}} \|u(s)\|_{L^2(Q)} 
&\leq   \limsup_{k\to \I} 	\sup_{0<s<T} \frac 1 {|Q|^{1/3}} \|u^{(k)}(s)\|_{L^2(Q)},
   \llabel{EQ107}   
}
which is uniformly bounded in $k$ by the  bounds for $u^{(k)}$ in
$\MM$.   For the gradient terms, the weak convergence implies
\[
\frac 1 {|Q|^{2/3}}\int_0^T\int_Q |\nabla u|^2\,dx\,dt\leq \liminf_{k\to \I} \frac 1 {|Q|^{2/3}} \int_0^T\int_Q |\nabla u^{(k)}|^2\,dx\,dt. 
\]
We now take a supremum over $Q\in \mathcal C$   to obtain the boundedness of $\alpha(T)$ and $\beta(T)$ for every $T>0$.

Given a bounded open set $\Omega \subset \HH$ we define $p$ via the local pressure expansion,
\[
p\coloneqq   p_{\li,\loc} + p_{\li,\nonloc} + p_{\loc,\Helm}+ p_{\loc,\harm}+ p_{\nonloc,\Helm}+ p_{\harm, \leq 1}+ p_{\harm, \geq 1},
\]
where we use the formulas \eqref{EQ20}, \eqref{EQ26}, \eqref{EQ32}, \eqref{EQ34}, and \eqref{EQ35} for each of the pressure parts. We similarly define $p^{(k)}$ as the local pressure expansion of each $u^{(k)}$, where $k\geq 1$, and we set
\[
\overline{p}^{(k)}= p^{(k)} - p.
\]
We also use analogous notation for the differences between each part of the local pressure expansion. We now claim that, for every compact set $K \subset \overline{ \Omega } $,

\begin{itemize}
\item the following parts converge to zero in $L^{\fract{3}{2}}(0,T;L^{\fract{3}{2}} (K))$: $\bar p^{(k)}_{\loc,\Helm}$, $ \bar p^{(k)}_{\nonloc,\Helm}$, $\bar p^{(k)}_{\harm, \geq 1}$, $\bar p^{(k)}_{\harm, \leq 1}$,
\item the following parts converge to zero in $L^{p}(0,T;L^2(K))$ for $p<\fract{4}{3}$: $\bar p^{(k)}_{\li,\loc}$,\, $\bar p^{(k)}_{\li,\nonloc}$.  
\item the sequence $\bar p^{(k)}_{\loc,\harm}$ is bounded in $L^{p}(0,T;L^{p} (K ))$ for every $p< 3/2$ and that $\| \bar p^{(k)}_{\loc,\harm}\|_{L^{\frac32}(0,T;L^{\frac32} (\tilde{K} ))} \to 0 $ as $k\to \infty$ for every $\tilde{K} \Subset \HH$ (i.e., locally away from the boundary).
\end{itemize}
These convergence properties guarantee that $(u,p)$ satisfies the Navier-Stokes equations \eqref{EQ01} in the sense of distributions on $\Omega$ as well as the local energy inequality \eqref{EQ155} for non-negative test functions $\phi\in C_0^\I(({\Omega}\cap \overline{\HH})\times [0,\I))$. Indeed, using the above convergence modes (of $u^{(k)}$ and of each part of the pressure function) the only nontrivial convergence is 
\[
\int_0^t \int  p_{\loc , \harm }^{(k)} u^{(k)} \cdot \nabla \phi \,dx\,ds\to \int_0^t \int p_{\loc , \harm } u \cdot \nabla \phi \,dx\,ds,
\]
for each $t>0$ and each nonnegative $\phi\in C_0^\I(\overline{\HH}\times [0,\I))$. For this, let $K\Subset \overline{\HH}$ be such that $\supp \phi (s) \subset K$ for all~$s$, fix $\epsilon >0$, and let $\tilde{K} \Subset \HH$ be such that 
\[\| u\|_{L^{\frac{19}6} ((K \setminus \tilde{K})\times (0,t))}\leq \epsilon \left( 2 \| (p_{\loc , \harm }^{(k)}, p_{\loc , \harm }) \|_{L^{\frac{19}{13}} (K\times(0,t) )} \right)^{-1}
.
\]
Then
\[\begin{split}
&\int_0^t \int_K \left|  p^{(k)} u^{(k)} - p u \right|\,dx\,ds \leq \int_0^t \int_K \left| u ( p^{(k)}  - p ) \right| \,dx\,ds+ \int_0^t \int_K \left|  (u^{(k)}-u ) p^{(k)}  \right|\,dx\,ds 
\\&\indeq
\leq \| u\|_{L^{\frac{19}6} ((K \setminus \tilde{K})\times (0,t))} \| (p^{(k)}, p) \|_{L^{\frac{19}{13}} }  + \| u \|_{L^{19/6}} \| p^{(k)} - p \|_{L^{\frac{19}{13}} (\tilde K\times(0,t) )} 
 + \| u^{(k)}- u \|_{L^{\frac{19}6}}\| p^{(k)} \|_{L^{\frac{19}{13}}}  
\\&\indeq
\leq \frac{\epsilon }2 + C \left( \| p^{(k)} - p \|_{L^{\frac{19}{13}} (\tilde K \times (0,t))}  + \| u^{(k)}- u \|_{L^{\frac{19}6}}\right) 
\\&\indeq
\leq \epsilon 
\end{split}
\]
for a sufficiently large $k$, as required, where for brevity we omitted the notation ``$\loc,\harm$'' and we set $L^q= L^q( K\times (0,t))$. 

Moreover, the local pressure expansion for $u$ (recall Definition \ref{D02}(6)) then follows for every open and bounded $\Omega \subset \HH$ by the uniqueness argument as in \eqref{EQ36}. The remaining properties of local energy solutions (i.e., that $u(t)\to u_0$ in $L^2_{\loc}$ and that $u(t)$ is weakly continuous in $L^2_{\loc}$) can be proven using well-known techniques, see e.g.~\cite{KiSe,KwTs}.\\

We now fix $K\Subset \HH$ and prove the convergence properties listed above. 

For $p_{\li,\loc }$ and $p_{\li,\nonloc}$, we have 
  \begin{equation}
   \| \bar p_{\li,\loc}^{(k)} (t) \|_{L^2 (K)}\lec_{K,Q}  \|u_0^{(k)}-u_0  \|_{L^2(Q^{*})}  t^{-\frac{3}{4}}
   \llabel{EQ108}  
  \end{equation}
and 
  \begin{equation}
   \| \bar p_{\li,\nonloc}^{(k)} (t)   \|_{L^\infty (K)}\lec_{K,Q}  \| u_0^{(k)}-u_0  \|_{M^{2,2}_{\mathcal C}}  t^{-\frac{3}{4}},
   \llabel{EQ109}  
  \end{equation}
as in Lemmas~\ref{L01} and~\ref{L02}. Since $u_0^{(k)}\to u_0\in   M^{2,2}_{\mathcal C}$, it follows that $p_{\li}^{(k)}\to p_{\li} \in L^{p} (0,T;L^2(Q))$ for every $T<\I$ and $p<\fract{4}{3}$.

The local part of the pressure is expanded into the harmonic and the Helmholtz parts. For the Helmholtz part, we have
  \begin{align}
  \begin{split}
   \int_0^T\| \bar p_{\loc,\Helm}^{(k)}(t)\|_{L^{\fract{3}{2}}(\HH)}^{\fract{3}{2}}\,dt &\lesssim \int_0^T \|(u^{(k)}\otimes u^{(k)}  - u\otimes u)(t)\|_{L^{\fract{3}{2}}(Q^{***})}^{\fract{3}{2}}\,dt
    \\&\lesssim \| (u,u^{(k)})\|^{\fract{3}{2}}_{L^3(0,T;L^3(Q^{***}))} \| u^{(k)}-u\|^{\fract{3}{2}}_{L^3(0,T;L^3(Q^{***}))},
  \end{split}
   \llabel{EQ110}  
  \end{align}
where the first inequality follows as in Lemma~\ref{L03}. The right-hand side vanishes as $k\to \I$ for every $T<\I$ by \eqref{EQ105}. Thus
$p_{\loc,\Helm}^{(k)}\to p_{\loc,\Helm}\in L^{3/2}(0,T;L^{3/2}(Q))$. \\

For $\overline{p}^{(k)}_{\loc,\harm}$, we have 
\eqnb\llabel{EQ111}
p_{\loc,\harm} (x,t) =p_A+p_B :=\frac{1}{2\pi i }\int_0^t \int_\Gamma \ee^{-(t-s)\la } \int_{\HH} q_{\la} (x'-z',x_3,z_3 ) \cdot ( F_A(z,s)+F_B(z,s) ) \d z \, \d \la \,\d s
  ,
\eqne
where, recalling \eqref{EQ32}, $F_A$ is a 2D vector function whose components are sums of terms of the form $  \p_j \left( \chi_{**} u_k u_l \right)=:  \p_j f_A $, where $j,k,l\in \{ 1,2,3 \}$, and $F_B(z,s)=F_B(z',z_3,s)$ is a sum of 2D~vectors of the form
  \[\begin{split}
&  m(D') \na' \otimes \na' \int_0^\infty  \left( \left( P(\cdot,|z_3 -y_3|) + P(\cdot,z_3+y_3) \right)\ast (\chi_{**} v\otimes w ) (y_3) \right)(z',s) \d y_3 
  \\&\indeq
   =: \nabla'  \otimes f_B(z,s)
   ,
  \end{split}
  \]
where $v$ and $w$ denotes various 2D vectors whose components are chosen among $u_1$, $u_2$, or $u_3$;
also,  $m(D')$ denotes a multiplier in the horizontal variable $z'$ that is homogeneous of degree $0$, and $\widehat{P}(\xi', t) \coloneqq \ee^{-t|\xi' |}$,
i.e., $P$ is the 2D~Poisson kernel.
Thus, using $\|m(D') P(\cdot , s) \|_{L^1 (\RR^2 )} \lec 1$,
which is a consequence of \eqref{EQ66} and $|y_3|\lec_Q 1$, we obtain
\[
\|f_B(z',z_3,s)\|_{L^p_{z'}(\RR^2)} \lec_Q \| \nabla' (\chi_{**} v\otimes w )(s) \|_{L^p(\HH)} 
\]
for every $z_3>0$, $s>0$, $p\geq 1$.
Thus, by \eqref{EQ23}, we get 
\eqnb\label{EQ112}\begin{split}
\| p_\BB (t) \|_{L^p (\HH)} &\lec \int_0^t \int_\Gamma \ee^{(t-s)\mathrm{Re}\,\lambda} \left\| \int_0^\infty \ee^{-|\lambda |^{\frac12} z_3} \left\| \int_{\RR^2 } (|x'-z'|+z_3+x_3)^{-3} f_B(z,s) \d z'\right\|_{L^p_{x'}(\RR^2)} \d z_3\right\|_{L^p_{x_3}}\!\!\!\d |\lambda | \,\d s\\
&\lec \int_0^t\| \nabla' (\chi_{**} v\otimes w )(s) \|_{L^p(\HH)}  \int_\Gamma \ee^{(t-s)\mathrm{Re}\,\lambda}  \int_0^\infty \ee^{-|\lambda |^{\frac12} z_3} \left\| (x_3+z_3)^{-1} \right\|_{L^p_{x_3}(0,\infty )} \d z_3\,\d |\lambda | \,\d s\\
&\lec \int_0^t\| \nabla' (\chi_{**} v\otimes w )(s) \|_{L^p(\HH)}  \int_\Gamma \ee^{(t-s)\mathrm{Re}\,\lambda}  \int_0^\infty \ee^{-|\lambda |^{\frac12} z_3} z_3^{-1+1/p} \d z_3\,\d |\lambda | \,\d s  \\
&\lec \int_0^t\| \nabla' (\chi_{**} v\otimes w )(s) \|_{L^p(\HH)}  \int_\Gamma \ee^{(t-s)\mathrm{Re}\,\lambda}  |\lambda |^{-\frac{1}{2p}}\,\d |\lambda | \,\d s\\
&\lec \int_0^t\| \nabla' (\chi_{**} v\otimes w )(s) \|_{L^p(\HH)}  (t-s)^{\frac{1-2p}{2p}} \,\d s,
\end{split}
\eqne
where we used $\int_{\RR^2} (|y'|+a)^{-3} \d y' \lec a^{-1}$ 
and $\int_0^\infty \ee^{-|\lambda |^{\frac{1}2}v} v^b \d v \lec |\lambda |^{-\frac{1+b}2}$. 
Thus 
\[
\| p_\BB  \|_{L^p ((\HH)\times(0,T))} \lec \| \nabla' (\chi_{**} v\otimes w )(s) \|_{L^1((0,T);L^p(\HH))} T^{-1+\frac{3}{2p}}
\]

Note that we have 
\eqnb\llabel{EQ113}\begin{split}\| \nabla' (\chi_{**} v\otimes w )(s) \|_{L^1((0,T);L^p(\HH))} &\lec_Q  \int_0^T \| u(s) \|_{L^{\frac{2p}{2-p}}}\|\nabla u (s) \|_{L^2(Q^{***})}\d s+ \| u \|_{L^2((0,T);L^{2p} (Q^{***}))}^2 \\
&\lec  \| u \|_{L^2((0,T);L^{\frac{2p}{2-p}} (Q^{***}))} \| \nabla u \|_{L^2((0,T);L^{2} (Q^{***}))}+ \| u \|_{L^2((0,T);L^{2p} (Q^{***}))}^2 ,
\end{split}
\eqne
which is bounded due to \eqref{EQ105} for $p<3/2$. Thus both $p_\BB $ and $p_\BB^{(k)}$, $k\geq 1$, are bounded in  $L^{p}(0,T;L^{p} (K ))$ for every $p< 3/2$. On the other hand, on a set $\tilde K=K'\times K_3$ compactly embedded in $\HH$
\[\begin{split}
\| \overline{p}^{(k)}_B (\cdot , x_3, t) \|_{L^{\frac{3}2}_{x'} (K')} &\lec  \int_0^t\| |u^{(k)}(s)|^2 - |u(s)|^2 \|_{L^{\frac32}(Q^{***})}  \int_\Gamma \ee^{(t-s)\mathrm{Re}\,\lambda}  \int_0^\infty \ee^{-|\lambda |^{\frac12} z_3} (x_3+z_3)^{-2}  \d z_3\,\d |\lambda | \,\d s\\
&\lec_K  \int_0^t\| |u^{(k)}(s)|^2 - |u(s)|^2 \|_{L^{\frac32}(Q^{***})}  \int_\Gamma \ee^{(t-s)\mathrm{Re}\,\lambda}  |\lambda |^{-\frac12} \,\d |\lambda | \,\d s\\
&\lec   \int_0^t\| |u^{(k)}(s)|^2 - |u(s)|^2 \|_{L^{\frac32}(Q^{***})}  (t-s)^{-\frac{1}2}\,\d s
\end{split}
\]
for every $x_3\in K_3$ and $t>0$, where the first inequality follows in the same way as the first two inequalities in \eqref{EQ112}, except that we now put both derivatives from $\nabla'\otimes \nabla '$ onto $q_\lambda$ (rather than one onto $q_\lambda$ and one onto $(\chi_{**} v\otimes w)$, and the second inequality follows simply by bounding $x_3+z_3\gtrsim_K 1$. Thus 
\[
\| \overline{p}^{(k)}_B  \|_{L^{\frac{3}2} ((0,T); L^{\frac32}  (K))} \lec_K T^{\frac12} \| |u^{(k)}(s)|^2 - |u(s)|^2 \|_{L^{{\frac{3}2}} ((0,T); L^{\frac{3}2}  (Q^{***}))}
\]
for every $T>0$, which vanishes in the limit $k\to \infty$ due to \eqref{EQ105}.

For $\overline{p}^{(k)}_A $, we have that $ f_A^{(k)}$ converges to $f_A$ in $L^{\frac{3}2} ((0,T);L^{\frac{3}2}(\HH)$ by \eqref{EQ105}; recall that $f_A$ is a sum of the term of the form $\chi_{**} u_lu_m$, $l,m=1,2,3$, and $f_A^{(k)}$ is defined analogously.  For every $t\in (0,T)$, $T>0$, and every $q\in (3/2 ,5/3)$,
\[\begin{split}
\| {p}_A (t) \|_{L^{\frac32} (K)} &\lec \int_0^t \int_\Gamma \ee^{(t-s)\mathrm{Re}\,\lambda} \left\| \int_0^\infty \ee^{-|\lambda |^{\frac12} z_3} \left\| \int_{\RR^2 } (|x'-z'|+z_3+x_3)^{-3} f_A (z,s) \d z'\right\|_{L^{\frac32}_{x'}(\RR^2)} \d z_3\right\|_{L^{\frac32}_{x_3}}\!\!\d |\lambda | \,\d s\\
&\lec_{K,Q} \int_0^t \int_\Gamma \ee^{(t-s)\mathrm{Re}\,\lambda} \left\| \int_0^\infty \ee^{-|\lambda |^{\frac12} z_3} (x_3+z_3)^{-1} \left\| f_A (\cdot ,z_3 ,s)\right\|_{L^{\frac32}_{z'}(\RR^2)} \d z_3\right\|_{L^{\frac32}_{x_3}(K_3)}\d |\lambda | \,\d s\\
&\lec_{K} \int_0^t \int_\Gamma \ee^{(t-s)\mathrm{Re}\,\lambda}  \int_0^\infty \ee^{-|\lambda |^{\frac12} z_3}z_3^{-\frac{1}{3}}  \left\| f_A (\cdot ,z_3 ,s)\right\|_{L^{\frac32}_{z'}(\RR^2)} \d z_3\d |\lambda | \,\d s\\
&\lec_{Q,q} \int_0^t  \left\| f_A (s) \right\|_{L^q(\HH)} \int_\Gamma \ee^{(t-s)\mathrm{Re}\,\lambda}  |\lambda |^{\frac{1}{6} -\frac{1}{2q'}}\d |\lambda | \,\d s\\
&\lec \int_0^t  \left\| f_A (s) \right\|_{L^q(\HH)} (t-s)^{\frac{1}{2q'} -\frac{7}{6}}  \,\d s,
\end{split}
\]
where $q'\in (5/2,3)$ is the conjugate exponent to $q$.
Therefore, we have the required estimate
\[
\| {p}_A  \|_{L^{\frac{3}2} ((0,T); L^{\frac32}  (K))} \lec_{K,Q,q} T^{\frac{1}{2q'} -\frac{1}{6}} \| f_A \|_{L^{{q}} ((0,T); L^{q}  (Q^{***}))}
\]
for any $q\in (3/2,5/3)$ (we can choose any such $q$), and similarly for ${p}^{(k)}_A$. By replacing $f_A $ by $f^{(k)}_A - f_A$ we also obtain the convergence $\| \overline{{p}}^{(k)}_A  \|_{L^{\frac{3}2} ((0,T); L^{\frac32}  (K))}\to 0$, as required. \\

The nonlocal components are $ \bar p^{(k)}_{\nonloc,\Helm}$, $\bar p^{(k)}_{\harm,\geq 1}$, and $\bar p^{(k)}_{\harm, \leq 1}$.  
The first of these is similar to the case of $\RR^3$ in \cite{BK,BKT}, but we include details to illustrate the main approach. Denote
\[
A_T = \sup_{t\in (0,T)} \sup_k \| (u(t),\,u^{(k)}(t))\|_{\MM}^2,\qquad B_T = \sup_k \sup_{Q\in\mathcal{C} } |Q|^{-\frac23} \int_0^T \int_Q \left( |\na u^{(k)}|^2 +|\na u |^2\right) .
\]

Recalling the details of the proof of Lemma~\ref{L05}, we have
  \begin{equation}\begin{split}
| \bar p^{(k)}_{\nonloc,\Helm}(x,t)|      &\lec_\Om \sum_{l>M } \sum_{\tilde Q\in S_l } \int_{\tilde Q} \frac {| u^{(k)}(z)\otimes u^{(k)}(z)-u(z)\otimes u(z)|} {|x-z|^4}\,\d z  + \int_{Q_M}| u^{(k)}\otimes u^{(k)}-u\otimes u|\,dz
    \\
    &\lec A_T		\sum_{l>M} 2^{-2l} 	+ \int_{Q_M}| u^{(k)}\otimes u^{(k)}-u\otimes u|\,\d z
  \end{split}   \llabel{EQ114}
  \end{equation}
  for every $x\in K$, where $M\in \NN $ is such that $2^M \gg \dist (\Om,0) $, and $Q_M = \bigcup_{k\leq M} \bigcup_{\tilde Q\in S_k} \tilde Q$.  Since the series above can be estimated by $2^{-2M}$, we can choose $M$ sufficiently large so that
  \begin{equation}\begin{split}
   \int_0^t \| \bar p^{(k)}_{\nonloc,\Helm}(s)\|_{L^\I(K)} \,ds  \leq \frac{\epsilon}2 + C_\Om \int_0^t \int_{Q_M(0)}| u^{(k)}\otimes u^{(k)}-u\otimes u|\,\d z\,\d s \leq \epsilon
   \llabel{EQ115}
  \end{split}\end{equation}
  for any preassigned $\epsilon$, where we have taken large $k$ in the last inequality. Due to the uniform bounds of both $p^{(k)}_{\nonloc,\Helm}$ and $p_{\nonloc,\Helm}$ in $L^\I(0,T;L^\I(\Om ))$, due to Lemma~\ref{L05}, we obtain the required convergence $ \bar p^{(k)}_{\nonloc,\Helm} \to 0$ in $L^{3/2}(0,T; L^{3/2}(K))$.

For $\overline{p}^{(k)}_{\harm, \geq 1}$, we write
\[
|\overline{p}^{(k)}_{\harm, \geq 1}| \leq  \overline{p}^{(k)}_{\AAA1} +  \overline{p}^{(k)}_{\AAA2} +  \overline{p}^{(k)}_{\AAA3} +  |\overline{p}^{(k)}_{\BB}|,
\] 
as in Lemma~\ref{L06}. As for $\overline{p}^{(k)}_{\AAA1}$ we have for $x\in K$ that 
\EQ{
&   |\bar p^{(k)}_{\AAA1}(x,t)|
\\&\indeq\lec_\Om  \int_0^t \int_\Ga \ee^{-(t-s)\RR e \lambda}  \underbrace{  \int_1^\infty \int_{\RR^2}\frac {\ee^{-|\la|^{\fract{1}{2}}z_3 }} {(|\xi'-z'|+\xi_3+z_3)^4} (1-\chi_*)|u^{(k)}\otimes u^{(k)}-u\otimes u| \,dz'\,dz_3}_{=:I}\,d|\la|\,ds,
   \llabel{EQ116}
}
where $\xi$ is on the line segment between $x$ and $x_\Om$, which is a fixed point inside $\Om$. As in Lemma~\ref{L06}, we have $\ee^{-|\la|^{\fract{1}{2}}z_3} \lec |\la |^{-\frac12}$, which gives 
  \begin{align}
   \begin{split}
   I &\lec_\Om |\la|^{-\fract{1}{2}} \sum_{l>M} \sum_{\tilde Q \in S_l} 2^{-4l } \int_{\tilde Q } (|u|^2 + |u^{(k)}|^2) + |\la|^{-\fract{1}{2}} \int_{Q_M} (1-\chi_*(z) )\frac  { |u^{(k)}\otimes u^{(k)}-u\otimes u|(z)} {(|\xi'-z'|+\xi_3+z_3)^4}  \,dz 
   \\&\lec_\Om |\la|^{-\fract{1}{2}}A_T  \sum_{l>M}  2^{-2l } + |\la|^{-\fract{1}{2}} \| u^{(k)}\otimes u^{(k)}-u\otimes u\|_{L^1(Q_M)}
   \\&\leq |\la|^{-\fract{1}{2}} A_T {2^{-2M}} + |\la|^{-\fract{1}{2}}\| |u^{(k)}\otimes u^{(k)}-u\otimes u|\|_{L^1(Q_M)}
   ,
   \end{split}
   \llabel{EQ117}
  \end{align}
for large $M$, where we used \eqref{EQ14} in the first inequality. Inserting this into the above integral in $|\la|$ and $s$ leads to  
  \begin{equation}
   \| \bar p_{\AAA1}^{(k)}(t)\|_{L^\I(\Om )}\lesssim_{\Om } t^{\fract{1}{2}} A_T   {2^{-2M}}  + \int_0^t (t-s)^{-\fract{1}{2}} \|u^{(k)}-u\|_{L^2(Q_M)}\|(u,u^{(k)})\|_{L^2(Q_M)}\,ds.
   \llabel{EQ118}  
  \end{equation}
  Thus, applying Young's inequality for the convolution in time we obtain
\[
\| \bar p_{\AAA1}^{(k)}\|_{L^2((0,T);L^\I(\Om ))}\lesssim_{\Om , T }   A_T  {2^{-2M}} + \sup_{0<s<T}\| (u (s) , u^{(k)}(s))\|_{L^2(Q_M)}  \|u^{(k)}-u\|_{L^2( Q_M \times (0,T))}
 ,
\]  
which converges to $0$ (by first choosing large $M$ and then large $k$).  

For $\bar p_{\AAA2}^{(k)}$, we have, as in \eqref{EQ57},
 \[
  \begin{split}
  |\bar p_{\AAA2}^{(k)}(x,t)| &\lec_\Om    \int_0^t (t-s)^{-\fract{9}{10}}  \int_{\RR^2} \int_0^1  \frac{ z_3^{-\fract{1}{5}}}{(|\xi'-z'| + \xi_3 +z_3 )^{4}}  (1-\chi_* ) |u^{(k)}\otimes u^{(k)}(z) - u\otimes u (z)|   \d z_3\, \d z'  \,\d s  ,
  \end{split}
  \]
  and, as in \eqref{EQ62}, we have for every $\tilde Q \in S_k$
  \eqnb 
  \begin{split}
    &   \int_0^1\int_{\tilde Q}  |u^{(k)}\otimes u^{(k)} - u\otimes u|z_3^{-\frac{1}{5}}dz'dz_3 \lec_{a} \| u^{(k)}\otimes u^{(k)} - u\otimes u \|_{L^{\frac32}(\tilde Q)}
   \\&\indeq
    \lec  2^{\fract{2}{3}k} \big(  \| (u,\,u^{(k)}) \|_{L^2(\tilde Q)} \|(\nabla u,\,\nabla u^{(k)})\|_{L^2(\tilde Q)} +2^{-k} \|(u,\,u^{(k)})\|_{L^2(\tilde Q)}^2   \big).
  \end{split}
   \llabel{EQ119}  
\eqne
Using Young's inequality for the convolution in time, we thus obtain
 \[\| \bar p_{\AAA2}^{(k)} \|_{L^{\frac32}((0,T);L^\I(K )) } \lec_{\Om, T} \left( (A_TB_T )^{\frac12} + A_T\right)  \sum_{m >M} 2^{-\frac43 m} + C_M \| u^{(k)}\otimes u^{(k)} - u\otimes u \|_{L^{\frac32}((0,T);L^{\frac32}(Q_M))} ,  
\]
where, for the first term, we used $(|\xi'-z'| + \xi_3 +z_3)\gtrsim 2^{k}$ for $z\in \tilde Q\in S_m$, $m>M$ and  
  \begin{equation}
   \int_0^T \|(\nabla u(t),\,\nabla u^{(k)}(t))\|_{L^2(\tilde Q)}^2 \lec 2^{2k} B_T   
   ,
   \llabel{EQ158}
  \end{equation}
for $\tilde Q\in S_m$, $m>M$; for the second term, we used the first step in the inequality above. We now choose a large $M$, and then a large $k$ to obtain the required convergence.

For $\bar p_{\AAA3}^{(k)}$ we have, as in \eqref{EQ54},
\[\begin{split} | \bar p_{\AAA3}^{(k)}(x,t) | &\lec_\Om  \int_0^t \int_\Gamma \ee^{-(t-s)\re\lambda} \int_{\supp\, \na \chi_*} \frac{\ee^{-|\la |^{\frac12}z_3}}{(|\xi' -z' | + \xi_3 + z_3)^4} |u^{(k)}\otimes u^{(k)}(z,s) - u\otimes u (z,s) |\d z  \d|\lambda| \d s\\
&\lec_\Om   \int_0^t \int_\Gamma \ee^{-(t-s)\re\lambda} |\la |^{-\frac16}  \d|\lambda| \|u^{(k)}\otimes u^{(k)}(s) - u\otimes u(s) \|_{L^{\frac32} (Q^{**})}  \d s\\
&\lec  \int_0^t (t-s)^{-\frac56}  \|u^{(k)}\otimes u^{(k)}(s)- u\otimes u(s) \|_{L^{\frac32} (Q^{**})}  \d s,
\end{split}
\]
which gives $\| \bar p_{\AAA3}^{(k)} \|_{L^{\frac32}((0,T);L^{\frac32}(K )) }\to 0$ by applying Young's inequality for the convolution in time.

As for $\bar p^{(k)}_B$, recalling \eqref{EQ68}, we write
 \[
  |\bar p^{(k)}_B (x,t)| \lec \int_0^t \int_\Gamma \ee^{(t-s)\re\la } \int_{\HH} \left| q_{\la , x, x_{\Om }} (z ) (1-\chi_*) \bar F_B^{(k)}(z,s) \right| \d z \, \d |\la| \d s ,
  \]
where $\bar F_B^{(k)}$ (defined in the same way as $F_B$ (recall Step 2 of the proof of Lemma~\ref{L06}), but with $u\otimes u$ replaced by $u^{(k)}\otimes u^{(k)}-u\otimes u$), can be estimated by 
\eqnb\label{EQ120}
|\bar F_B^{(k)}( z) | \lec_Q 2^{-3M} A_T + C_M \| u^{(k)}\otimes u^{(k)}- u\otimes u \|_{L^1(Q_M)}
\eqne
for any fixed $Q\subset \HH$, and $z\in Q$, where $M>0$ large enough so that $Q\subset Q_{M/2}$. This can be obtained by an easy modification of Step 2 of the proof of Lemma~\ref{L06} by separating the integration region into $Q_M$ and the rest, as above. We then obtain
\[\begin{split}
 |\bar p^{(k)}_B (x,t)| &\lec_Q \int_0^t \int_\Gamma \ee^{(t-s)\re\la } |\la |^{-\frac12} \left( \sum_{l>L} \sum_{\tilde Q\in S_l} 2^{-2l} \|\bar F_B^{(k)} (s) \|_{L^\infty (\tilde Q)} + C_L \|\bar F_B^{(k)} (s) \|_{L^\infty (Q_L )} \right) \d |\la| \d s ,\\
& \lec \int_0^t (t-s)^{-\frac12}  \left( A_T \sum_{l>L} 2^{-4l}  + C_L A_T 2^{-3M} + C_L C_M \| u^{(k)}\otimes u^{(k)}(s)- u\otimes u(s) \|_{L^1(Q_M)} \right)  \d s ,
 \end{split}
\]
where, in the second line, we have used the estimate $\|\bar F_B^{(k)} (s) \|_{L^\infty (\tilde Q)}\lec |\tilde Q|^{-\frac23} A_T$ (from Step 2 of the proof of Lemma~\ref{L06}) in the summation, and have assumed that $M>2L$ in order to use \eqref{EQ120}. This gives 
\[
\| \bar p_{\BB}^{(k)} \|_{L^{\frac32}((0,T);L^{\infty}(K )) } \lec_{\Om,T} A_T 2^{-4L} + C_L A_T 2^{-3M} + C_L C_M \| u^{(k)}\otimes u^{(k)}- u\otimes u \|_{L^{\frac32}((0,T);L^{\frac32}(Q_M))} ,
\]
which provides the required convergence by first choosing large $L$, then $M$ and $k$.

Finally, for the remaining component of the nonlocal pressure,   $\bar p^{(k)}_{\harm, \leq 1}$, we have 
  \begin{equation}\begin{split}
\left|   \bar p^{(k)}_{\harm, \leq 1} (x,t)\right| &= \frac{1}{2\pi } \left| \int_0^t \int_\Gamma \ee^{(t-s)\la } \int_{\HH} q_{\la } (x'-z', x_3,z_3 ) \chi_* \bar F^{(k)}_B(z,s)' \d z \, \d \la \d s \right|\\
&\lec_\Om  \int_0^t \int_\Gamma \ee^{(t-s)\mathrm{Re}\,\la }  \| \bar F^{(k)}_B(s) \|_{L^\infty (Q^*)} \int_{Q^* } \left| q_{\la } (x'-z', x_3,z_3 )\right| \d z \, \d \la \d s 
   \llabel{EQ121}
  \end{split}\end{equation} 
for every $x\in K$. Thus noting that $\int_{Q^* } \left| q_{\la } (x'-z', x_3,z_3 )\right| \d z \lec_{Q^*} \int_0^\infty \ee^{-|\la |^{\frac12} z_3} z_3^{-\frac{1}2} \d z_3 \lec |\la |^{-\frac14}$ (recall \eqref{EQ23} and \eqref{EQ41}) and using \eqref{EQ120} gives
  \[
  \begin{split}
\left|   \bar p^{(k)}_{\harm, \leq 1} (x,t)\right| &\lec_\Om  \int_0^t (t-s)^{-\frac34} \left( A_T 2^{-3M } + C_M  \| u^{(k)}\otimes u^{(k)}(s)- u\otimes u(s) \|_{L^1(Q_M)}\right) \d s ,
\end{split}  \]
 which implies the convergence $\|  \bar p^{(k)}_{\harm, \leq 1} \|_{L^{\frac32}((0,T);L^\infty (K))}\to 0$, as required.
\end{proof}

\section{Existence}
\label{sec06}

In this section we apply the stability result of Section \ref{sec05} to obtain the existence of global weak solutions when $u_0\in   \MC$ and of a scaling invariant solutions when $u_0$ is scaling invariant.
We start with the generic case.

\begin{proof}[Proof of Theorem~\ref{T01}]
Let $u_0^{(k)}\in C_0^\infty (\overline{\HH})$ be divergence-free and such that $u_0^{(k)}\to u_0\in M^{2,2}_{\mathcal C}$. By the Maekawa-Miura-Prange theory, we obtain global local energy solutions $u^{(k)}$ and pressure $p^{(k)}$. 
Following \cite[Proposition 3.1]{MMP1}, these solutions satisfy the local pressure expansion tailored to $\mathcal C$ and $\mathcal C_n$.
Hence, these solutions satisfy the a priori bounds in Lemma~\ref{T02}. The asserted global solution exists due to Theorem~\ref{T03}.
\end{proof}

We now address Theorem~\ref{T06}. As our foundation, we use the scaling invariant solutions of \cite{BT2}. These belong to the energy perturbed class which we now recall.

\begin{Definition}[EP-solutions to \eqref{EQ01}]\label{D04} The vector field $u$ defined on $\HH\times (0,\I)$ is an \emph{energy perturbed solution to \eqref{EQ01}}, abbreviated `{EP-solution},' with divergence free initial data $u_0\in L^{3,\I}(\HH)$ if 
\begin{equation}\llabel{EQ151} 
\int_0^\I  \big( (u,\partial_s f)-(\nabla u,\nabla f)- (u\cdot\nabla u ,f)  \big)  \,ds =0,
\end{equation}
 for all $f\in  \{ f\in C_0^\I(\HH\times (0,\infty)):\nabla\cdot f = 0 \}$, we have
\[
u-Su_0 \in L^\infty(0,T;L^2(\HH))\cap L^2(0,T;H^1(\Omega)),
\]
for any $T>0$, and 
  \begin{equation}
   \lim_{t\to 0^+} \| u(t)-Su_0(t) \|_{L^2(\HH)} = 0,
   \llabel{EQ13}  
  \end{equation}
where $Su_0(t)\in L^\I(0,\I;L^{3,\I}(\HH))$ is the solution to the time-dependent Stokes system with initial data $u_0$ and zero boundary value.
\end{Definition}

The main theorem of \cite{BT2} is the following.
  
\begin{Theorem}[\cite{BT2}] \label{T05}
If $u_0\in  L^{3,\I}(\HH)$  is SS (resp.~$\lambda$-DSS) and such that $u_0|_{x_3=0} = 0$, then there exists an EP-solution $u$ on $\HH\times [0,\infty)$ with initial data $u_0$,  which is SS (resp.~$\lambda$-DSS). Moreover, $u|_{x_3=0}(x,t) =0$ for $t>0$. 
\end{Theorem}
  
  Here (and below) we understand the boundary condition $u_0|_{x_3=0} = 0$ in the sense that for any $x\in \partial \HH\setminus \{ 0 \}$, there exists a neighborhood of $x$ so that $u_0=0$ in this neighborhood. Our proof of Theorem~\ref{T06} is by stability using the solutions of Theorem~\ref{T05} as approximations. 
To connect these with a scaling invariant datum in $\MC$ we need the following lemma.

\begin{Lemma}\label{L12}
Assume $u_0\in \MC$ is $\lambda$-DSS and such that $u_0|_{x_3=0} = 0$. Then there exists a sequence $\{u_0^{(k)}\}\subset L^{3,\I}(\HH) $ so that $u_0^{(k)}|_{x_3=0}=0$, all $u_0^{(k)}$ are $\la$-DSS and $u_0^{(k)} \to u_0$ in $\MM$. If $u_0$ is self-similar, then $u_0^{(k)}$ can also be taken to be self-similar.
\end{Lemma}

The proof of this is similar to the proof of \cite[Lemma~4.1]{BT5} and the details are omitted.

\begin{proof}[Proof of Theorem~\ref{T06}]
Concerning the solutions of Theorem~\ref{T05}, it is an easy exercise to check they are local energy solutions. Indeed, the component $S$ is smooth and decays in the sense that it belongs to $\mathcal L^2_{\uloc}$ where we adopt the notation of \cite{MMP1}. The $L^2$ part also enjoys this decay. This is a sufficient condition for $u$ to have the local pressure expansion which follows by adapting \cite[Proof of Proposition 3.1]{MMP1} to $\mathcal C$ and $\mathcal C_n$.  Based on this, the solutions from Theorem~\ref{T05} clearly satisfy the conditions of Theorem~\ref{T03}. 

Given $u_0$, from Lemma~\ref{L12} we obtain a sequence $u_0^{(k)}$ which converges to $u_0$ in $\MM$. By Theorem~\ref{T05} we obtain for each $k$ a global EP-solution $u^{(k)}$. These solutions and data satisfy the assumptions of Theorem~\ref{T03}. Hence, there exists a local energy solution $u$ for initial data $u_0$ which is a limit of $u^{(k)}$ in the sense given in the proof of Theorem~\ref{T03}. This convergence is sufficient to guarantee $u$ is DSS.  The argument is identical when $u$ is self-similar. 
\end{proof}

\section{Eventual Regularity}
\label{sec07}

The goal of this section is to prove Theorem~\ref{T04}, which asserts
regularity in a parabolic region with an arbitrarily small leading coefficient.

Assume that $0<q\leq 1$. Also, suppose that 
  \begin{equation}
   u_0\in \mathring M_{{\mathcal C}}^{2,q}
   ,
   \llabel{EQ122}
  \end{equation}
and assume that $u$ is a solution as constructed in the previous sections.
Denote
  \begin{equation}
  \alpha_n (t) =\sup_{s\in[0,t]}\| u (s) \|_{M_{\mathcal{C}_n}^{2,q}}^2
  \quad \text{ and } \quad 
  \beta_n(t) = \sup_{Q\in \mathcal{C}_n} \frac{1}{|Q|^{\fract{q}{3}}} \int_0^t \int_{Q} |\na u |^2
   .
   \llabel{EQ123}
  \end{equation}
In this section, we mostly omit indicating the dependence of $\alpha$ and $\beta$ on~$n$.

In the proof of the main theorem, we shall use the following version of the Gronwall lemma.

\cole
\begin{Lemma}\label{L13}
Suppose that 
$f\colon [0,T_0]\to [0,\infty)$ is a  nonnegative increasing 
continuous
function,
which satisfies
  \begin{equation}
    f(t) 
     \leq a f(0) 
     + a \left(1+\frac{t}{b}\right)^\gamma
      \left( 
        \Vert f \Vert_{L^\infty (0,t)}^{p_1} 
        +         \Vert f \Vert_{L^\infty (0,t)}^{p_2} 
        + c     \Vert f \Vert_{L^\infty (0,t)}
     \right),
    \llabel{EQ124}   
  \end{equation}
where  $p_1, p_2>1$ and
$\gamma,b,c\geq0$.
For every $K\geq1$, there
exists $\epsilon>0$, 
depending on $a$, $b$, $\gamma$, $p_1$, $p_2$, and $K$,
such that
if
  \begin{equation}
   f(0) + c
   \leq \epsilon
   ,
   \llabel{EQ125}
  \end{equation}
then
  \begin{equation}
   f(t) \leq 2 a f(0)
   ,
   \llabel{EQ126}
  \end{equation}
for $t\leq \min\{K b,T_0\}   $.
\end{Lemma} 
\colb

\begin{proof}[Proof of Lemma~\ref{L13}]
The proof is obtained by the barrier argument, comparing the solution
$f(t)$ with $2 a f(0)$.
\end{proof}

Another important ingredients in the proof of Theorem~\ref{T04} is the following estimate on the pressure term in the energy inequality.

\cole
\begin{Lemma}
\label{L14}
Let $q\in (0,1)$, $n\in {\mathbb N}$. If $u_0$ and $u$ are as above, then
  \begin{align}
   \begin{split}
    \frac{1}{|Q|^{\frac{q}{3}}}
    \int_0^t \int p \,u \cdot \na \phi_Q 
    \lec_q
    \bigl(1+ t|Q|^{-\frac23}\bigr)^{C} 
    \Bigl(
    \alpha_n(0)^{\frac12} \alpha_n(t)^{\frac12}
    +
       (\alpha_n(t)+\beta_n(t))^{\frac32}
    +
       (\alpha_n(t)+\beta_n(t))^{3}
   \Bigr)
   ,
   \end{split}
   \label{EQ127}
  \end{align}
for all $Q\in {\mathcal C}_n$.
\end{Lemma}
\colb

\begin{proof}[Proof of Lemma~\ref{L14}]
We set
  \begin{equation}
   \tt =\frac{t}{|Q|^{2/3}}
   .
   \llabel{EQ128}
  \end{equation}
From \eqref{EQ84}, we have
  \begin{align}
  \begin{split}
        \frac{1}{|Q|^{\frac{q}{3}}}
    \int_0^t \int (p_{\li,\loc} +p_{\li,\nonloc})u\cdot \na \phi_Q 
&    \lec
      |Q|^{-\frac16}
      t^{\fract{3}{16}} 
      \| u_0 \|_{M_{\cn}^{2,q}} 
     \| \alpha \|_{L^8(0,t)}^{\frac{1}{2}}
     \\&
      \lec
    |Q|^{-\frac1{24}}      \tilde t^{\fract{1}{4}} 
      \| u_0 \|_{M_{\cn}^{2,q}}
      \alpha(t)^{\frac{1}{2}}
   .
   \end{split}
   \llabel{EQ129}
  \end{align}
By \eqref{EQ85}, we have
  \begin{equation}\begin{split}
    \frac{1}{|Q|^{\frac{q}{3}}}
  \int_0^t \int p_{\loc,\Helm} \,u\cdot \na \phi_Q 
   &\lec
   C_\epsilon |Q|^{\frac{2q}{3}-\fract{4}{3}}  \| \alpha \|_{L^3 (0,t)}^3  
        +\epsilon \beta(t) 
        + |Q|^{\fract{q}{6}-\fract{5}{6}} \| \alpha \|_{L^{\frac32} (0,t)}^{\frac32}  
   \\&
   \lec
   C_\epsilon |Q|^{\frac{2q}{3}-\fract{2}{3}} \tt  \alpha(t)^3   
        +\epsilon \beta(t) 
        + |Q|^{\fract{q}{6}-\fract{1}{6}} 
           \tt
          \alpha(t)^{\frac32}  
   .
  \end{split}
   \llabel{EQ130}
  \end{equation}
Next, by \eqref{EQ86}, 
  \begin{equation}\begin{split}
  &
   \frac{1}{|Q|^{\frac{q}3}}
   \int_0^t \int p_{\loc,\harm} \,u\cdot \na \phi_Q 
   \\&\indeq
   \lec   |Q|^{\frac{q-2}{6}}
     \left(  
             t^{\frac{253}{1632}}\| 
             \alpha \|_{L^{8}(0,t)}^{\frac{25}{68}}  
             \beta(t)^{\frac9{68}} 
           + |Q|^{-\fract{3}{34}} 
             t^{\frac{13}{48}}
             \| \alpha \|_{L^{8}(0,t)}^{\frac12}   \right) 
   \\&\indeqtimes
    \left(
     t^{\frac{1}{816}}
      \| \alpha \|_{L^{8}(0,t)}^{\frac{13}{34}} \beta(t)^{\frac{21}{34}}
      + 
       |Q|^{-\frac{4}{51}}
       t^{\frac{5}{48}}
       \| \alpha \|_{L^8 (0,t)}^{\frac12}   
       \beta(t)^{\frac12}
      +
        |Q|^{-\frac{21}{51}} 
        t^{\frac{13}{24}}
        \| \alpha \|_{L^{8}(0,t)}  
      \right) 
      \\&\indeq
      \leq  |Q|^{\frac{q-2}{6}} (\alpha(t)+\beta (t))^{\frac32}
     \left(  
             t^{\frac{41}{205}}
           + |Q|^{-\fract{3}{34}} 
             t^{\frac{1}{3}}
              \right) 
    \left(
     t^{\frac{5}{102}}
           + 
       |Q|^{-\frac{4}{51}}
       t^{\frac{1}{6}}
            +
        |Q|^{-\frac{21}{51}} 
        t^{\frac{2}{3}}
            \right) 
      \\&\indeq
      =  |Q|^{\frac{q-2}{6}} (\alpha(t)+\beta (t))^{\frac32}
     \left(  
          |Q|^{\frac2{15}}   {\tilde t}^{\frac{41}{205}}
           + |Q|^{\fract{41}{306}} 
             {\tilde t}^{\frac{1}{3}}
              \right) 
   |Q|^{\frac{5}{153}}   \left(
   {\tilde t}^{\frac{5}{102}}
           + 
{\tilde       t}^{\frac{1}{6}}
            +
          {\tilde     t}^{\frac{2}{3}}
            \right) 
      \\&\indeq
      \lec  |Q|^{\frac{q-1}{6}}(1+\tilde t) (\alpha(t)+\beta (t))^{\frac32},
  \end{split}
\colb
   \llabel{EQ131}
  \end{equation}
  where we used that $\alpha (t)$ is nondecreasing in the second inequality, and the fact that $\frac2{15}+\frac5{153}<\frac16$ in the last one.
  
From \eqref{EQ88}, we have
  \begin{equation}\begin{split}
    & 
   \frac{1}{|Q|^{\frac{q}3}}
    \int_0^t \int p_{\nonloc,\Helm} \,u\cdot \na \phi_Q 
      \lec
      |Q|^{\frac{q}6-\fract{5}{6}} 
      \Vert \alpha \Vert_{L^{\frac32}(0,t)}^{\frac32}  
     \lec
      |Q|^{\frac{q}6-\fract{1}{6}} 
      \tt
      \alpha(t)^{\frac32}  
  .
   \end{split}
   \llabel{EQ132}
  \end{equation}
For $p_{\harm, \geq 1}$ we take $r=3/2$ in Lemma~\ref{L06} to obtain
\eqnb\label{EQ133}\begin{split}
  \| p_{\harm, \geq 1} \|_{L_t^{\frac{3}{2}}L_x^\infty (Q\times(0,t))} &\lec_{\gamma, \delta } |Q|^{\frac{q-3}3} (1+t)^{\gamma} \left( |Q|^{\frac{4\delta}9 }\| \alpha \|_{L^{\frac{3(1-\delta )}{2-3\delta }}(0,t)}^{1-\delta }  \beta(t)^{\delta } +(1+t^{\frac12} |Q|^{-\frac13}  ) \| \alpha \|_{L^{\frac32} (0,t)}  \right)\\
  &\lec |Q|^{\frac{q-3}3} (1+t)^{\gamma} \left( |Q|^{\frac{4\delta}9 }t^{\frac{2-3\delta }3 } \alpha(t)^{1-\delta }  \beta(t)^{\delta } +(1+t^{\frac12} |Q|^{-\frac13}  ) t^{\frac23} \alpha(t)  \right)\\
  &\lec |Q|^{\frac{q-3}3} (\alpha(t) + \beta (t) )(1+t)^{\gamma} t^{\frac23}  \left( 1+ |Q|^{\frac{4\delta}9 }t^{-\delta }  +t^{\frac12} |Q|^{-\frac13}    \right)
  \end{split}
  \eqne
for any $q\in (0,3)$, $\gamma \in (0,1)$ and $\delta \in (0, \min \{ 2/3 , 3q/2 \})$. This gives 
  \begin{align}
   \begin{split}
     &  
    \frac{1}{|Q|^{\frac{q}{3}}}
    \int_0^t \int_{Q} p_{\harm, \geq 1} \,u \cdot \na \phi_Q 
    \lec
    |Q|^{-\frac{q}{3}-\frac{1}{3}}
     \| p_{\harm, \geq 1} \|_{L_t^{\frac{3}{2}}L_x^\infty (\Om\times(0,t))}
     \| u \|_{L_t^{3}L_x^1 (\Om\times(0,T))}
    \\&\indeq
    \lec
    |Q|^{-\frac{q}{3}-\frac{1}{3}}
      |Q|^{\frac{q-3}3} (\alpha(t) + \beta (t) )(1+t)^{\gamma} t^{\frac23}  \left( 1+ |Q|^{\frac{4\delta}9 }t^{-\delta }  +t^{\frac12} |Q|^{-\frac13}    \right)
    |Q|^{\frac12+\frac{q}{6}}
   \Vert \alpha\Vert_{L^{\frac32}(0,t)}^{\frac12}
    \\&\indeq
      \lec
    |Q|^{\frac{q}{6}-\frac{5}{6}}
      (\alpha(t) + \beta (t) )^{\frac32}(1+t)^{\gamma} t \left( 1+ |Q|^{\frac{4\delta}9 }t^{-\delta }  +t^{\frac12} |Q|^{-\frac13}    \right)
           \\&\indeq
    =
    |Q|^{\frac{q-1+4\gamma}{6}}
      (\alpha(t) + \beta (t) )^{\frac32}(|Q|^{-\frac23}+\tilde t)^{\gamma} {\tilde t} \left( 1+ |Q|^{-\frac{2\delta}9 }{\tilde t}^{-\delta }  +{\tilde t}^{\frac12}   \right)
           \\&\indeq
    \lec_q  (\alpha(t) + \beta (t) )^{\frac32} (1+\tilde{t})^2,
   \end{split}
   \label{EQ134}
  \end{align}
where, in the last step, we have chosen $\gamma\in (0,1/2)$ sufficiently small so that $q-1+4\gamma <0$ and $\delta \coloneqq 2q/3$.
Finally, we use \eqref{EQ90} to obtain
  \begin{equation}\begin{split}
  &
  \frac{1}{|Q|^{\frac{q}{3}}}
  \int_0^t \int_{Q} p_{\harm,\leq 1} \,u\cdot \na \phi_Q 
  \lec  
    |Q|^{\frac{q}6-\frac{13}{12}} t^{\frac{19}{16}}  \| \alpha \|_{L^8 (0,t)}^{\frac32}
   \lec
    |Q|^{\frac{q-1}6} 
    \tt^{\frac{11}{8}}  
    \alpha(t)^{\frac32}
   .
  \end{split}
   \llabel{EQ135}
  \end{equation}
Summing the above inequalities, we obtain \eqref{EQ127}, as required.  
\end{proof}

Before the proof, we also need the following fact.

\cole
\begin{Lemma}
\label{L15}
Let $u$ be as above. Then
$\alpha+\beta$ is a continuous function of $t$.
\end{Lemma}
\colb

\begin{proof}[Proof of Lemma~\ref{L15}]
The proof is similar to \cite[Proof of Lemma~3.2]{BK}.
We only sketch the  continuity of $\alpha$ 
on $[0,T]$, where $T>0$ is fixed,
as the argument for
$\beta$ is simpler.
First, for every $Q\in{\mathcal C}$,
$\sup_{s\in[0,t]}\int_{Q}|u|^2\phi$ is a continuous function of $t$.
The rest follows by
  \begin{equation}
   \lim_{n\to \infty}
    \sup_{Q\in{\mathcal C}; |Q|\geq 2^{n}}
    \frac{1}{|Q|^{\frac{q}{3}}} \sup \int_{Q}|u(0,T)|^2\phi
    = 0
   \llabel{EQ136}
  \end{equation}
by finding $n\in {\mathbb N}$ such that
$T_n$ in Theorem~\ref{T02} satisfies $T_n\geq T$
and by applying Theorem~\ref{T02} with $n\to\infty$.
\end{proof}

We are now ready to prove the main theorem on eventual regularity.

\begin{proof}[Proof of Theorem~\ref{T04}]
From Remark~\ref{R01}, recall that
  \begin{equation}
     \frac 1 {|Q|^{\frac23}}\int_0^{t}\int_Q |u|^3\,dx\,dt 
       \lec
         t^{\frac{1}{4}} 
         |Q|^{\fract{q}{2}-\frac23}
         \alpha(t)^{\fract{3}{4}}
         \beta(t)^{\fract{3}{4}}
       +
       t |Q|^{\frac{q}{2}-\frac76}
         \alpha(t)^{\fract{3}{2}}
      .
   \label{EQ137}
  \end{equation}
for all $Q\in {\mathcal C}_n$, from where
  \begin{align}
   \begin{split}
   \frac{1}{|Q|^{\frac{q}{3}}}
    \int_0^t \int_{Q} |u|^2 u \cdot \na \phi_Q 
    \lec
         t^{\frac{1}{4}} 
         |Q|^{\fract{q-2}{6}}
         \alpha(t)^{\fract{3}{4}}
         \beta(t)^{\fract{3}{4}}
       +
        t |Q|^{\frac{q-5}{6}}
        \alpha(t)^{\fract{3}{2}}
   .
   \end{split}
   \label{EQ138}
  \end{align}
Note that the both terms on the right-hand side are dominated by the right-hand
side of \eqref{EQ127}. Thus, 
applying \eqref{EQ127} and \eqref{EQ138} in the
energy inequality \eqref{EQ95}, we get
  \begin{align}
   \begin{split}
   \alpha(t) + \beta(t)
   &\les
   \alpha(0)
    +
    \bigl(1+ t|Q|^{-\frac23}\bigr)^{C} 
    \Bigl(
    \alpha(0)^{\frac12} \alpha(t)^{\frac12}
    +
       (\alpha(t)+\beta(t))^{\frac32}
    +
       (\alpha(t)+\beta(t))^{3}
   \Bigr)
   \end{split}
   \llabel{EQ139}
  \end{align}
since also the term $   t |Q|^{-\frac23} \alpha(t)$ is dominated by the right-hand side of \eqref{EQ127}.
Note that $|Q|^{-\frac23}\les 2^{-2n}$. 
With $\epsilon>0$ to be determined, find $n\in {\mathbb N}$
sufficiently large such that
  \begin{equation}
   \alpha(0)
   \leq \epsilon
   .
   \llabel{EQ140}
  \end{equation}
(Recall that $\alpha(t)=\alpha_n(t)$.)
Now, we apply Lemma~\ref{L13} with $f(t)=\alpha(t)+\beta(t)$,
which is continuous by Lemma~\ref{L15}.
Thus we obtain
  \begin{equation}
   \alpha(t) + \beta(t)
   \les
   \epsilon
   \comma 
   0\leq t\leq 2^{n}
   ,
   \label{EQ141}
  \end{equation}
and then \eqref{EQ137} implies
  \begin{equation}
   \frac{
    1
       }{
    2^{2n}
   }
   \iint_{Q\times(0,2^{n})}
    |u|^{3}
    \les C \epsilon
   ,
   \label{EQ142}
  \end{equation}
where
  \begin{equation}
   Q= \left[
         -2^{n},2^{n}
       \right]^{3}
     \cap \HH
    .
   \llabel{EQ143}
  \end{equation}
To apply the boundary regularity criterion, we also need to obtain an
estimate for $2^{-2n}\iint_{Q\times(0,2^{n})} |p|^{3/2}$.
Due to integrability issues at time 0 of the estimates in \eqref{EQ37}, we 
bound the pressure on the space-time cylinder
$Q\times(1,2^{n})$ rather than $Q\times(0,2^{n})$. 
For the linear part of the pressure, $p_{\li,\loc}+ p_{\li,\nonloc} $, we use H\"older's inequality and the first two inequalities in \eqref{EQ37} to get 
 \eqnb \begin{split}
   \frac{1}{|Q|^{\frac{2}{3}}}
   \| p_{\li,\loc}+ p_{\li,\nonloc} \|_{L^{\frac32} ((1,t); L^{\frac{3}{2}} (Q))}^{\frac32}
   \lec
   |Q|^{\frac{3q-10}{12}} 
   \Vert u_0\Vert_{M^{2,q}}^{\frac32}
   \lec \epsilon
   .
   \end{split}
   \label{EQ144}
\eqne
Next, by Lemma~\ref{L03} and \eqref{EQ142}, we get
  \begin{align}
   \begin{split}
   \frac{1}{|Q|^{\frac{2}{3}}}
   \| p_{\loc,\Helm} \|_{L^{\frac32} ((0,t); L^{\frac{3}{2}} (Q))}^{\frac32}
   \lec  \frac{1}{|Q|^{\frac{2}{3}}}
   \| u \|_{L^{3} ((0,t); L^{3} (Q^{***}))}^{3} \lec  \epsilon
   .
   \end{split}
   \llabel{EQ145}
  \end{align}
For the local harmonic part of the pressure, we have, with $t=2^{2n}$,
  \begin{align}
   \begin{split}
   &
   \frac{1}{|Q|^{\frac{2}{3}}}
   \| p_{\loc,\harm} -\theta\|_{L^{\frac32} ((0,t); L^{\frac{3}{2}} (Q))}^{\frac32}
    \lec
      |Q|^{-\frac{28}{51}}
   \| p_{\loc,\harm} -\theta\|_{L^{\frac32} ((0,t); L^{\frac{17}{10}} (Q))}^{\frac32}
   \\&\indeq
    \lec
    |Q|^{-\frac{28}{51}}   |Q|^{\frac{q}2} \left(
       \| \alpha \|_{L^{\fract{39}{5}}(0,t)}^{\fract{13}{34}} \beta(t)^{\fract{21}{34}}+ |Q|^{-\frac4{51}}\| \alpha \|_{L^3 (0,t)}^{\fract12}   \beta(t)^{\fract12}+|Q|^{-\frac{21}{51}} \| \alpha \|_{L^{\fract32}(0,t)}   \right)^{\frac32}
    ,
    \\ &\indeq
    \lec 
    |Q|^{\frac{q}2-\frac{28}{51}}   \left(
       t^{\frac{5}{102}} + |Q|^{-\frac4{51}} t^{\frac{1}6}+|Q|^{-\frac{21}{51}} t^{\frac23}  \right)^{\frac32} (\alpha (t) + \beta (t))^{\frac32}
    ,
    \\ &\indeq
    \sim
    |Q|^{\frac{q-1}2}   (\alpha (t) + \beta (t))^{\frac32}
    ,
    \\ &\indeq
    \lec \epsilon
    ,
   \end{split}
   \llabel{EQ146}
  \end{align}
where we used Lemma~\ref{L04} in the second inequality,  the fact $t=2^{2n}\sim  |Q|^{\frac23}$ in the fourth inequality, and \eqref{EQ141} in the last one.

For the nonlocal Helmholtz pressure, we use Lemma~\ref{L05} to obtain
 \begin{align}
   \begin{split}
   &
   \frac{1}{|Q|^{\frac{2}{3}}}
   \| p_{\nonloc,\Helm} \|_{L^{\frac32} ((0,t); L^{\frac{3}{2}} (Q))}^{\frac32}
    \lec |Q|^{-\frac23} |Q|^{\frac{q-1}2} \| \alpha \|_{L^{\frac32}(0,t)}^{\frac32}\lec  |Q|^{\frac{q-1}2} (\alpha (t) )^{\frac32} \lec \epsilon
   \end{split}
   \llabel{EQ147}
  \end{align}
by \eqref{EQ141}.
Moreover, \eqref{EQ133} gives
 \begin{align}
   \begin{split}   
   \frac{1}{|Q|^{\frac{2}{3}}}
   \| p_{\harm ,\geq 1 } \|_{L^{\frac32} ((0,t); L^{\frac{3}{2}} (Q))}^{\frac32}
    &\lec |Q|^{\frac13} |Q|^{\frac{q-3}2} (\alpha(t) + \beta (t) )^{\frac32}(1+t)^{\frac{3\gamma}{2}} t  \left( 1+ |Q|^{\frac{4\delta}9 }t^{-\delta }  +t^{\frac12} |Q|^{-\frac13}    \right)^{\frac32}
   \\&
   \lec |Q|^{\frac{q}2-\frac76}   (\alpha(t) + \beta (t) )^{\frac32}|Q|^{\gamma} |Q|^{\frac23}   \left( 1+ |Q|^{-\frac{2\delta}9 }    \right)^{\frac32}
   \\&
   \lec |Q|^{\frac{q-1}2 + \gamma }   (\alpha(t) + \beta (t) )^{\frac32}
   \\&
   \lec \epsilon,
   \end{split}    
   \llabel{EQ148}
  \end{align}
where we used \eqref{EQ141} and the choice of $\gamma <(1-q)/4$, as in \eqref{EQ134} above, in the last inequality. 

Finally, Lemma~\ref{L07} gives
  \begin{align}
   \begin{split}
   \frac{1}{|Q|^{\frac{2}{3}}}
   \| p_{\harm, \leq 1} \|_{L^{\frac32} ((0,t); L^{\frac{3}{2}} (Q))}^{\frac32}\lec |Q|^{\frac13} |Q|^{\frac{3q-15}8} \alpha (t)^{\frac32} \int_0^t s^{\frac9{16}} \d s \lec  |Q|^{\frac{3q-4}8} \alpha (t)^{\frac32} \lec \epsilon
   \end{split}
   \label{EQ149}
  \end{align}
using $t\lec |Q|^{\frac23}$ in the second inequality and \eqref{EQ141} in the last.

Summing the estimates \eqref{EQ144}--\eqref{EQ149}, we conclude that there exists
$\theta=\theta(t)$ such that
  \begin{align}
   \begin{split}
   &
   \frac{1}{2^{2n }} 
   \| p -\theta\|_{L^{\frac32} ((1,2^{2n}); L^{\frac{3}{2}} (Q))}^{\frac32}
   \lec \epsilon
   .
   \end{split}
   \label{EQ150}
  \end{align}
Let $\epsilon_0\in(0,1)$ be arbitrary.
Applying
the boundary regularity criterion due to Seregin~et~al~\cite[Theorem 1.1]{SSS}
to the inequalities
\eqref{EQ142} and \eqref{EQ150}, with $\epsilon$ sufficiently small,
we obtain that
$(u,p)$ is regular in the set
$(-(1-\epsilon_0)2^{n},(1-\epsilon_0)2^{n})^{3}\times (1-(1-\epsilon_0)^2) 2^{2n},2^{2n})$, and the theorem follows.
Note that the criterion \cite{SSS} requires
the integrability condition
$\na p, D^{2}u \in L_{\loc}^{\frac32}( (0,\infty); L^{\frac98}_{\loc } (\overline{\HH}))$,
which is the content of Lemma~\ref{L16} stated next.
\end{proof}

\cole
\begin{Lemma}[Integrability of $\na p$ and $D^{2}u$]
\label{L16}
For any local energy solution $(u,p)$, we have
  \begin{equation}
   \na p, D^{2}u \in L^{\frac32}( (t_0,T); L^{\frac98}_{\loc } (\overline{\HH}))
   ,
   \label{EQ83}
  \end{equation}
 for every $T>0$ and $t_0\in (0,T)$.
\end{Lemma}
\colb

\begin{proof}[Proof of Lemma~\ref{L16}]
Let $\Omega \subset \HH$ be a bounded open set. We review the pressure estimates \eqref{EQ37} and observe that the gradients of each nonlocal part of the pressure, i.e., $\na p_{\li, \loc}$, $\na p_{\nonloc, \Helm }$, $\na p_{\harm, \geq 1}$, belong to $L^{\frac32} (t_0,T); L^{\frac98} (\Om ))$, as the derivative falling onto the kernel $q_\lambda$ only improves the estimate (since we obtain faster decay on the pointwise bound on $q_\lambda$ \eqref{EQ23}). For the local parts, we obtain the required regularity by reexamining their estimates from \eqref{EQ37}, as follows.

For $p_{\li,\loc }$, we argue as in \eqref{EQ39}--\eqref{EQ42}, with the $L^2$ norms replaced by $L^{\frac98}$,
to obtain
\[
\begin{split}
\| \na p_{\li, \loc } (t) \|_{L^{\frac98}(\Om )} &\lec \int_\Ga \ee^{t\re\la } \int_0^\infty \ee^{-|\la |^{\fract{1}{2}}z_3} \| \chi u_0 (\cdot , z_3) \|_{L^{\frac98}_{z'}} \left( \int_0^{c(\Om )} (x_3+z_3)^{-\frac98} \d x_3 \right)^{\frac89}  \d z_3 \d |\la |\\
&\lec_{\Om } \| \chi u_0  \|_{L^2} \int_\Ga \ee^{t\re\la } \left( \int_0^\infty \ee^{-2|\la |^{\fract{1}{2}}z_3} z_3^{-\frac29}  \d z_3\right)^{\frac12}  \d |\la |\\
&= \| \chi u_0  \|_{L^2} \int_\Ga \ee^{t\re\la } |\la |^{-\frac7{36}}  \d |\la |\\
&\lec \| \chi u_0  \|_{L^2} t^{-\frac{29}{36}}
,
\end{split}
\]
where we used the Cauchy-Schwarz inequality in the second inequality.

For $p_{\loc,\harm }$, we argue as in \eqref{EQ40} to obtain
\[
\| \na p_{\loc, \harm }\|_{L^{\frac32}((0,T);L^{\frac98} (\HH))}  \lec  \| \mathbb{P} \na \cdot (\chi_{**} u\otimes u ) \|_{L^{\fract32} ((0,T);L^{\fract{9}{8}}(\HH))} \lec_\Om \alpha(T) + \beta (T)
,
\]
where the dependence on $\Om$ is via the cutoff function $\chi_{**}$.

The estimate on $\na p_{\loc, \Helm}$ follows by direct calculation
and Calder\'on-Zygmund estimates, and the estimate on $\na p_{\harm ,
\leq 1}$ follows in the same way as $\na p_{\li , \loc }$ above, by
observing that $\| \chi_* F_B (t) \|_{L^2} \lec_\Om \alpha(T)$ for
every $t\in (0,T)$ (which can be obtained in the same way as Step 2 of
the proof of Lemma~\ref{L06}) and by integration in time. This 
gives $ \| \na p_{\harm, \leq 1} (t) \|_{L^{\frac98}(\Om )} \lec_\Om
\alpha(T) t^{\frac7{36}}$, and so  $\na p_{\harm, \leq 1} \in
L^{\frac32} ((t_0,T); L^{\frac98} (\Om ))$, as required.

In order to get the integrability assertion for $D^2u$, let
$\phi\in C_0^{\infty}(\overline{\HH}\times (0,\infty))$ be arbitrary.
Then we have
  \begin{equation}
    \partial_{t}(u \phi)
     - \Delta (u \phi)
     = f
    ,
   \label{EQ04}
  \end{equation}
where $f=-\phi u\cdot \nabla u - \phi\nabla p+ u (\partial_{t} \phi - \Delta \phi) - 2 \nabla \phi\cdot \nabla u$.
By the first part of the proof, we have 
$\phi\nabla p\in L^{\frac32}( ([0,\infty); L^{\frac98} (\overline{\HH}))$.
Also, $\phi u\cdot \nabla u\in L^{\frac32}( (0,\infty); L^{\frac98} (\overline{\HH}))$ since
$u\in L_{\loc}^{6}( (0,\infty); L^{\frac{18}{7}} (\overline{\HH}))$
and \\
$\nabla u\in L_{\loc}^{2}( (0,\infty); L^{2} (\overline{\HH}))$.
Using also the local square integrability of $u$ and $\nabla u$, we get
$f\in L^{\frac32}( [0,\infty); L^{\frac98} (\overline{\HH}))$.
Applying the maximal parabolic regularity (see section D.5 in \cite{RRS}, for example) to the equation \eqref{EQ04} with zero initial data,
we obtain $ D^2 (u\phi) \in L^{\frac32}( [0,\infty); L^{\frac98} (\overline{\HH}))$.
By local square integrability of $u,\nabla u$,
this implies
$ \phi D^2 u \in L^{\frac32}( [0,\infty); L^{\frac98} (\overline{\HH}))$,
and since $\phi$ was an arbitrary test function supported in 
$\overline{\HH}\times (0,\infty)$,
the proof is complete.
\end{proof}

\colb
\section*{Acknowledgments}
ZB and WO were supported in part by the Simons Foundation,
while IK was supported in part by the
NSF grant DMS-1907992.

\colb

\end{document}